\documentclass[11pt]{article}

\usepackage{geometry}
\usepackage{amsmath}
\usepackage{amssymb, amsfonts, amsthm}
\usepackage{tikz-cd}
\usepackage{enumerate}
\usepackage{mathrsfs}
\usepackage{graphicx}
\usepackage{xcolor}
\usepackage{hyperref}
\usepackage{subcaption}
\usepackage{rotating}
\usepackage{titlesec}

\usepackage[
backend=biber,
style=numeric,
url=false, 
doi=false
]{biblatex}
\renewbibmacro{in:}{}
\DeclareFieldFormat[article]{title}{\mkbibitalic{#1}}
\DeclareFieldFormat[article]{journaltitle}{#1\isdot}

\newcommand{\mf}{\mathfrak}

\newcommand{\mc}{\mathcal}

\newcommand{\C}{\mathbb{C}}
\newcommand{\Q}{\mathbb{Q}}

\newcommand{\R}{\mathbb{R}}
\newcommand{\Z}{\mathbb{Z}}

\newcommand{\PP}{\mathbb{P}}
\newcommand{\Hom}{\operatorname{Hom}}
\newcommand{\End}{\operatorname{End}}
\newcommand{\Ext}{\operatorname{Ext}}

\newcommand{\on}{\operatorname}
\newcommand{\Id}{\operatorname{Id}}

\newcommand{\Aut}{\operatorname{Aut}}

\newcommand{\Heis}{\mc{H}eis}

\newcommand{\csth}{{\C^\ast_\hbar}}

\newcommand{\wt}{\on{wt}}
\newcommand{\Res}{\on{Res}}
\newcommand{\pt}{\mrm{pt}}
\newcommand{\Jac}{\on{Jac}}
\renewcommand{\tilde}{\widetilde}
\newcommand{\mbb}{\mathbb}

\newcommand{\pd}[2]{\frac{\partial #1}{\partial #2}}
\newcommand{\GL}{\operatorname{GL}}

\newcommand{\mrm}{\mathrm}
\newcommand{\vac}{|\rangle}

\DeclareMathOperator{\NS}{NS}
\DeclareMathOperator{\SL}{SL}
\DeclareMathOperator{\Stab}{Stab}
\DeclareMathOperator{\Coh}{Coh}

\DeclareMathOperator{\NE}{NE}

\DeclareMathOperator{\gr}{gr}
\DeclareMathOperator{\Pic}{Pic}
\DeclareMathOperator{\Sym}{Sym}
\DeclareMathOperator{\A}{\mathbb{A}}

\titleformat{\subsubsection}[runin]{\bfseries}{\thesubsubsection}{0.5em}{}

\newtheorem{thm}{Theorem}[section]

\newtheorem{prop}[thm]{Proposition}

\theoremstyle{definition}
\newtheorem{defn}[thm]{Definition}

\newtheorem{rmk}[thm]{Remark}

\makeatletter
\let\c@equation\c@thm
\makeatother
\numberwithin{equation}{section}

\title{\vspace{-2cm} Elliptic surfaces and toroidal superalgebras}
\author{Samuel DeHority}
\begin{document}
\maketitle
\begin{abstract}
The equivariant cohomology of certain moduli spaces of sheaves on isotrivial elliptic surfaces are shown to admit representations of infinite dimensional Lie (super)algebras. The construction is based on work of Billig and Chen-Li-Tan on vertex representations of toroidal extended affine Lie algebras, and a geometric interpretation of $bc$-ghost systems on Hilbert schemes of points on ruled surfaces. The role of higher-nullity toroidal extended affine lie algebras are discussed in relation to tensor products of modules of a Yangian-like associative algebra. Applications to enumerative geometry are discussed in the introduction. 
\end{abstract}

\tableofcontents

\section{Overview}
This paper is the first in a series that calculates enumerative invariants of moduli spaces of sheaves on a class of elliptic surfaces using methods from geometric representation theory and noncommutative geometry. In this overview, we explain the main results of the current paper as well as their relevance to the entire series. 

Associated to a semisimple Lie algebra $\mf g$ with root system $R$, a (classical) toroidal Lie algebra 
$\widehat{\widehat{\mf g}}$ 
is a multivariable analogue of the affine Lie algebra 
$\widehat{\mf g}$ defined by 
\[ \widehat{\widehat{\mf g}} = \mf g\otimes  A \oplus \Omega^1_{ A}/d\Omega^0_{ A}\] 
where $ A = H^0(\mc O_{\C^{\ast 2}_{s,t}})$ is the ring of functions on the torus $\C^{\ast 2}_{s,t}$ and the bracket is defined by extending linearly from 
\[ [x f(s,t), y g(s,t)] = [x,y]f(s,t)g(s,t) + \langle x, y\rangle g df \] 
such that $\Omega^1_{ A} / d\Omega^0_{ A}$ is central. There are a number of Lie algebras containing $\widehat{\widehat{\mf g}}$, including the full toroidal algebra, the elliptic Lie algebra and the toroidal extended affine Lie algebra, whose roots may be labelled by the same root lattice but differ by root multiplicities. 

 The purpose of this paper is  to assign 
  a Lie algebra $\mf g_{X_R}$ and an associative algebra 
 $\mc A_\hbar(\mf g_{X_R})$ to an elliptic surface 
 $X_R$ 
 labelled by a root system $R$ encoding its singular fibers and to study the representation theoretic properties of these algebras. We expect that the associative algebra $\mc A_\hbar(\mf g_{X_R})$ admits a filtration so that 
 \[ \gr \mc A_\hbar(\mf g_{X_R}) = \overline{\mc U(\on{Maps}(\C, \mf g_{X_R}))}\] 
 making $\mc A_\hbar(\mf g_{X_R})$ a filtered deformation of some completion of the current algebra and hence an analogue of the Yangian. By different methods it is possible to construct an algebra $\mc A_K(\mf g_{X_R})$  which acts in topological $K$-theory and we expect $\gr A_{K}(\mf g_{X_R})$ is a filtered deformation of a completion of 
 \[ \mc U(\widehat{\mf g_{X_R}}) \supset \mc U (\on{Maps}(\C^{\ast 3}, \mf g_{R}) )\] 
containing a $3$-toroidal algebra associated to the finite dimensional Lie algebra $\mf g_R$.

 The Lie algebra  $\mf g_{X_R}$ contains the corresponding toroidal algebra $\widehat{\widehat{\mf g}}$ and  $\mc A_\hbar(\mf g_{X_R})$  is defined
   such that the enumerative invariants of $X_R$-fibered threefolds, or equivalently of moduli spaces of sheaves on $X_R$, should be able to be expressed in terms of the representation theory of $\mc A_\hbar(\mf g_{X_R})$. This is an elliptic surface analogue of the fact that enumerative geometry of Nakajima quiver varieties associated to a quiver $Q$ can be calculated using the quiver Yangian $Y(\mf g_Q)$ \cite{Maulik_Okounkov_2019} associated to a Lie algebra $\mf g_Q$ containing the Kac-Moody algebra whose Dynkin diagram is that quiver. 

Unlike the situation for quiver varieties, where there is a unified story for all quivers $Q$, the geometric, enumerative and representation theoretic results of this project are concerned with the case that $R$ lies in an exceptional series of root systems 
\begin{equation}\label{eq:exceptional_series}
  R \in \{ A_{-1}, A_0, A_1, A_2, D_4, E_6, E_7, E_8 \}
\end{equation} 
which is closely related to the more familiar Deligne exceptional series, which is discussed together with an explanation of the nonstandard root systems $A_{-1}$ and $A_0$ in Section \ref{ssec:root_systems} below. The Deligne series has important relations to Painlev\'e equations, vertex algebras and twistor theory. The Lie algebra $\mf g_{X_R}$ like the toroidal algebra $\widehat{\widehat{\mf g_R}}$ and a few related algebras, has root spaces labelled up to multiplicity by the elliptic root system 
\[ R^{ell} = \{ \beta + m [\pt] + n[E] \mid \beta \in R\cup \{0\}, ~m,n\in \Z \} \subset R \oplus K_{num}(E) \] 
where $E$ is an elliptic curve and $K_{num}(E) = K_0(E)/{\on{rad} \chi(-,-)}\simeq \Z^2 $ is the numerical Grothendieck group. 

In the present context, the series \eqref{eq:exceptional_series} arises from the condition on the structure of the singular fiber of an elliptic surface over $\A^1$ which admits an action of the multiplicative group scaling the base of the fibration. The series \eqref{eq:exceptional_series} corresponds to the fiber types
\begin{equation}
  \label{eq:kodaira_series}
  \{ \RN{1}_0, \RN{2},\RN{3} ,\RN{4}, \RN{1}_0^*,\RN{4}^*,\RN{3}^* ,\RN{2}^*  \} 
\end{equation}
in Kodaira's classification of singular fibres of minimal elliptic surfaces. 

There are two families and six isolated $\C^\ast_\hbar$-equivariant elliptic surfaces over $\A^1$ where $ \C^\ast_\hbar\subset T$ scales the base $\A^1$ with a unique fixed point at $0$ and scales the coordinate $b \in \A^1$ by a positive multiple of $\hbar$.  Here $T$ is a larger torus that also acts, potentially trivially, on moduli spaces under consideration. We label a choice of such surface by $X_R$ indexed by a Dynkin type in \eqref{eq:exceptional_series}. The root systems other than $A_{-1}$ in the exceptional series \eqref{eq:exceptional_series} organize into dual pairs 
\begin{equation}
  \label{eq:conj_pairs} (A_0, E_8), (A_1, E_7), (A_2, E_6), (D_4, D_4) 
\end{equation}
so that if $(H_1, H_2)$ are simply connected groups of type corresponding to a dual pair then $H_2$ is conjugate to the centralizer of $H_1$ in $E_8$ under the chain of embeddings 
\[ A_0 \subset A_1 \subset A_2 \subset D_4 \subset E_6 \subset E_7 \subset E_8.
 \] 
Geometrically, the dual pairs \eqref{eq:conj_pairs} correspond to the fibers over $0$ and $\infty \in \PP^1$ of a minimal elliptic surface $\overline{X_R}$ fitting into a diagram \[% https://tikzcd.yichuanshen.de/#N4Igdg9gJgpgziAXAbVABwnAlgFyxMJZABgBpiBdUkANwEMAbAVxiRAA0B9AJRAF9S6TLnyEUARnJVajFmwA68iDRgAnBljAxgXbn36CQGbHgJEy46fWatEIRQEEAeuINCToopMvVrcu4oACoEu-NIwUADm8ESgAGaqEAC2SGQgOBBIAEwC8YkpiJLpmYgAzLkgCcnZ1BlIpbV0WAxsABYQEADWbpX5qbUlRThNLXbtXWF8QA
\begin{tikzcd}
X_R \arrow[d] \arrow[r, hook] & \overline{X_R} \arrow[d] \\
\A^1 \arrow[r, hook]          & \PP^1                   
\end{tikzcd}\] 
so that if $(R_1, R_2)$ are a dual pair then $\overline{X_{R_1}} \simeq \overline{X_{R_2}}$ over an automorphism of $\PP^1$ exchanging $0$ and $\infty$. 

In the simplest case of the empty root system, the surface is $X_{A_{-1}} = T^*E = E\times\A^1$. The $A_{-1}$ case is the main example of the present paper and has a number of distinguishing features, including the existence of an odd part of the algebra on account of odd cohomology. Thus we will be primarily concerned with representation theoretic aspects of the algebras $\mf g_{T^*E}$ and $\mc A_\hbar(\mf g_{T^*E})$. The algebra $\mf g_{T^*E}$ has another description as $\widehat{\mc Ham}(\C^{\ast 2 | 2 })$ which is a central extension of the algebra of Hamiltonian vector fields on an affine group superscheme. Precisely, let $\mf g_{T^*E} = \bigoplus_{r,d} \mf g_{r,d}\oplus \Q \pmb{c}_s \oplus \Q \pmb{c}_t$ where $\mf g_{r,d}\simeq H^*(E)$.  Then $\mf g_{T^*E}$ is a Lie super algebra with bracket 
\begin{equation}\label{eq:main_bracket} [w^{a,b}_{\gamma}, w^{c,d}_{\gamma'}] = -(ad-bc) w^{a + c, b + d}_{\gamma \star \gamma'}  + \delta_{a+c, 0} \delta_{b + d, 0}\langle \gamma, \gamma'\rangle  (a\pmb{c}_s + b\pmb{c}_t)
\end{equation}
where $\star:  H^*(E) \otimes H^*(E) \to H^*(E)$ is the convolution product. The even Lie subalgebra of this Lie superalgebra is closely related to a toroidal extended affine Lie algebra, which extends the toroidal algebra by adjoining a subset of the derivations on the underlying torus.

The Lie bracket \eqref{eq:main_bracket} has a formulation using generating series. Letting $H^*(E) = \Q [\sigma_+, \sigma_-]$ be a ring generated freely by two odd generators and letting $Z = (z_1, z_2 | \psi_+ , \psi_-)$ denote a formal coordinate on a $2|2$-dimensional superspace define the generating series  $\mbb{D}(Z) \in \mf g_{T^*E}[ \psi_+ , \psi_-][[ z_1, z_2, z_1^{-1}, z_2^{-1}]]$ by 
\[ \mbb{D}(Z) = \sum_{a,b\in \Z} w^{a,b}_{1}\psi_+\psi_-z_2^{-a}z_1^{-b}+  w^{a,b}_{\sigma_-}\psi_+z_2^{-a}z_1^{-b}+  w^{a,b}_{\sigma_+}\psi_-z_2^{-a}z_1^{-b} + w^{a,b}_{\sigma_+\sigma_-}z_2^{-a}z_1^{-b}\] 
then the bracket \eqref{eq:main_bracket} is equivalent to 
\begin{multline}
  \label{eq:22_notation_D_bracket}
  [\mbb{D}(Z) ,\mbb{D}(W)] = D_{w_1}\mbb{D}(W + \psi)\delta\left( \frac {w_1}{z_1}
  \right)D_{w_2}\delta\left( \frac{w_2}{z_2} \right) -D_{w_2} \mbb{D}(W + \psi)D_{w_1}\delta\left( \frac {w_1}{z_1}
  \right)\delta\left( \frac{w_2}{z_2} \right)  \\
  + \pmb{c}_{t} D_{w_1} \delta\left( \frac {w_1}{z_1}
  \right)\delta\left( \frac{w_2}{z_2} \right)\delta\left( \frac {\phi_1}{\psi_1}
  \right)\delta\left( \frac{\phi_2}{\psi_2} \right)+ \pmb{c}_{s} \delta\left( \frac {w_1}{z_1}
  \right)D_{w_2}\delta\left( \frac{w_2}{z_2} \right)\delta\left( \frac {\phi_1}{\psi_1}
  \right)\delta\left( \frac{\phi_2}{\psi_2} \right)
\end{multline}
where $D_{w_i} = w_i \pd{}{w_i}$, and the delta function of a pair of odd variables $\theta, \chi$ is denoted $\delta(\theta/\chi) = \theta + \chi$. 

The existence of two-variable generating functions is related to the existence of a discrete group of automorphisms of $\mf g_{X_R}$ which acts by $\GL(2, \Z)$ on the coordinates $z_1, z_2$. Geometrically this discrete group arises from derived (anti)-autoequivalences of the surface $\overline{X_R}$. 

\subsection{Geometric vertex algebra realization}

Following a body of literature constructing vertex realizations of toroidal algebras and their extensions \cite{Moody_Rao_Yokonuma_1990,Eswara_Rao_Moody_1994,Billig_2006, Billig_2007,Chen_Li_Tan_2021} this paper constructs representations of $\mf g_{T^*E}$ using the theory of vertex operator algebras using previously established relations between vertex operator algebras and moduli spaces of sheaves on surfaces. 

The main geometric object considered in this paper is the Hilbert scheme of points $X_R^{[n]}$. Its equivariant cohomology $H^*_T(X_R^{[n]})$ we expect to serve as a weight space for a representation of $\mc A_\hbar(\mf g_{X_R})$ and its subalgebra $\mf g_{X_R}$. 

Specifically, outside of the main example of type $R = A_{-1}$,  in Theorem \ref{thm:goofy_other_surfaces} we only produce an action of a Lie algebra $\mf g_{X_R}$ on a larger vector space whose weight spaces are $H^*_T(\overline{X_R}^{[n]})$. The algebra $\mf g_{X_R}$ coincides with the toroidal extended affine Lie algebra of type $R$ when $R \neq A_{-1}$. The action on cohomology depends on a choice of a cohomology class $[\mc C] \in H^*_T(\overline{X}_R)$ represented by a cycle supported on $D_\infty = \overline{X_R}\backslash X_R$. We view this construction as preliminary although we expect the action to coincide with one defined by convolution by correspondences with relevance to enumerative geometry after pulling back to moduli spaces associated to $X_R$ itself. 

For this reason we focus on the case $R = A_{-1}$ where the vector space underlying the vertex algebra structure coincides with the cohomology of moduli spaces of points avoiding the compactification divisor. 

The starting point for the vertex construction is the construction 
\cite{Nakajima_1999,Grojnowski_1996,Nakajima_1996_jack}
 of a representation of the Heisenberg-Clifford algebra $\mc Heis_{H^*_T(S)} = \Q\langle \alpha_k(\gamma), \pmb{c} \rangle_{k \in \Z\backslash 0,\gamma \in B}$ modelled on on the equivariant cohomology $H^*_T(S)$ with basis $B$ over $H^*_T(\pt)$ of a semiprojective surface $S$ on the cohomology of the union of Hilbert schemes 
 \begin{equation}
  \label{eq:nakajima_hilb}\mc Heis_{H^*_T(S)} \to \End(\bigoplus_{n > 0}  H^*_T(S^{[n]})) 
\end{equation}
identifying the latter with the vacuum module 
\[ \mc F_S := \mc Heis_{H^*_T(S)}\vac \simeq \Q[\alpha_k(\gamma)]_{k <0, \gamma \in B}\] 
for a vertex algebra $(\mc F_S, Y, \vac, \omega)$. 

The toroidal algebra action and its relation to Hamiltonian vector fields is conceptually distinct from the action of the cohomology-labelled W super algebra $\mc W(H^*_T(X_R))$ which acts on $\mc F_S$ under the specialization $\hbar = 0$ by \cite{Li_Qin_Wang_2002, } and admits a filtration whose associated graded $\gr \mc W(H^*_T(X_R))$ is generated by $L^{a,b}_\gamma$ with commutation relations 
\[ [L^{a,b}_\gamma, L^{c,d}_\eta] = (ad-bc)L^{a+c-1, b+d}_{\gamma \cup \eta}   - \frac{\Omega_{b,d}^{a,c}}{12} L^{ a+ c-3, b + d}_{e\cup\gamma\cup \eta} + \cdots \]
where $\Omega_{b,d}^{a,c}$ is an integer and $\cdots$ has central terms.
It is an important question to describe concisely the algebra generated by $\mc W(H^*_T(X_R))$ and elements of $\mf g_{X_R}$ although the commutation relations may be calculated for any pair of generating fields using Wick's theorem to determine OPEs. More generally, we expect a close relationship between the action of tautological classes and slope subalgebras of $\mf g_{X_R}$ with the deformed W algebras which arise in \cite{Hausel_Mellit_Minets_Schiffmann_2022}.

A related question is to understand the relationship between $\mc A_\hbar(\mf g_{X_R})$ and the three-loop algebras associated to the Deligne series studied in \cite{Costello_Gaiotto_2021,Costello_Gaiotto_2021} and to explain the relationship between these algebras and the loop algebra of algebra of Hamiltonian vector fields appearing e.g. in \cite{Bittleston_2023}. 

Letting $D_\infty = \overline{X_{A_{-1}}} \backslash X_{A_{-1}} \simeq E$ and letting $\mc F \in \Coh(D_\infty)$ be a stable sheaf with coprime rank $r> 0$ and degree $d$ we will choose a polarization $H$ in a specific chamber for a wall-and-chamber structure for Gieseker stability on $\overline{X_{A_{-1}}}$ and define $M_{H}(v + m[E] + n [\pt]; \mc F )$ to be the moduli space of $H$-Gieseker stable sheaves $\mc E$ on $\overline{X_{A_{-1}}}$ equipped with a framing $\mc E|_{D_\infty} \simeq \mc F$. Then summing over an appropriate range of chern classes, define 

\begin{equation}
  \label{eq:VtE_intro_def} V_{{T^*E}, \mc F} := \bigoplus_{m,n} H_T^*(M_{H}(v + m[E] + n [\pt]; \mc F )).
\end{equation}

One of the main theorems of this series is the following. 

\begin{thm}\label{thm:main_nonlog_thm}
There is a vertex operator algebra structure $(V_{{T^*E}, \mc F}, Y(-, z), \vac, \omega)$ inducing a representation $\mf g_{T^*E} \to \End(V_{T^*E, \mc F})$. 
\end{thm}

There are analogous vertex operator algebras $V_{X_R}$ for any $R$ from \eqref{eq:exceptional_series}.  Owing to the multiplicative derivatives and delta functions in \eqref{eq:22_notation_D_bracket} the Fourier coefficients of $\mbb{D}(Z)$ do not act as Fourier modes of vertex operators in $V_{T^*E, \mc F}$ but following the terminology of \cite{Chen_Li_Tan_2021} they are identified with Fourier modes of a $\phi$-coordinated module structure of $V_{T^*E, \mc F}$ on itself. Explicit formulas for the representation are given in \eqref{eq:d_freefld}-\eqref{eq:sm_freefld} expressed in terms of the Heisenberg algebra from \eqref{eq:nakajima_hilb}.  

Pick an identification $E \simeq \Jac(E)$ and $\mc P$ the universal line bundle. There is a relative version of the moduli space $M_H(v; \mc F)$ of the form $M_H(v; \mc F\otimes \mc P)/\Jac(E)$. 
There is a related vertex algebra $V^{\Jac}_{T^*E,\mc F}$ with weight spaces in \eqref{eq:VtE_intro_def} replaced by their relative versions so that 

\begin{equation}
  \label{eq:VtE_intro_def} V^{\Jac}_{{T^*E}, \mc F} := \bigoplus_{m,n} H_T^*(M_{H}(v + m[E] + n [\pt]; \mc F\otimes \mc P )/\Jac(E)).
\end{equation}

There is a notion of a vertex algebra over a supercommutative ring \cite{Barron_2010}. Also based on work on logarithmic conformal field theory \cite{Gurarie_1993} one may define a \emph{logarithmic vertex algebra}  \cite{Bakalov_Villarreal_2022} in which fields and operator product expansions may have logarithmic singularities. Based on this there is a superalgebra $\mf g^{\Jac}_{T^*E}$ which contains 
\[\mf g_{T^*E}\otimes_{H^*_T(\pt)}H^*_T(\Jac(E))\]  defined in \ref{ssec:jac_g} which is used in the proof of \ref{thm:main_nonlog_thm} through the following result. 

\begin{thm}\label{thm:main_log_thm}
There is a conformal logarithmic vertex operator algebra structure on 
$V^{\Jac}_{{T^*E}}$ 
giving a representation of
 $\mf g^{\Jac}_{T^*E}$ 
 on 
 $V^{\Jac}_{T^*E,\mc F}$. 
\end{thm}

A key point is the identification of the zero modes and logarithmic modes of certain vertex operators in 
$V^{\Jac}_{T^*E}$ 
with geometric operations of cup product and convolution on the cohomology of the Jacobian. This aspect admits an immediate generalization to higher genus curves $C$ giving a geometric interpretation of a rank 
$n$ $bc$-system 
$V_{bc}^{\otimes n}$ 
in terms of moduli spaces of framed rank $1$ sheaves on ruled surfaces over $C$.

\subsection{Moduli of framed sheaves on noncommutative surfaces}

In this paper, Theorems \ref{thm:main_nonlog_thm} and \ref{thm:main_log_thm} are proven in their entirety for $\mc F$ of rank 1, while for higher rank they depend on an analysis of moduli spaces of framed objects $(\mc E, \phi) \in M_{\sigma}(v; \mc F)$ in the derived category $D^b(\Coh(\overline{X}_R))$ which are stable with respect to a Bridgeland stability condition $\sigma
\in \Stab(\overline{X}_R)$. That analysis is carried out in the second part of this series \cite{dehority_toroidal_springer}. 

There is a related problem of the identification of the geometric meaning of the operators on $V_{T^*E}$ corresponding to the generators of $\mf g_{T^*E}$. While the generators of the Heisenberg algebra arise from convolution with specific correspondences between two Hilbert schemes and the expression of the $w_\gamma^{a,b}$ operators in terms of locally finite expressions in $\mc Heis_{H^*_T(S)}$ in principle provides a geometric interpretation of the generators,  for applications to enumerative geometry a different description is necessary. 

For $A_n$ surfaces $\widetilde {\C^2/\Gamma}$ the identification of Fourier modes of lattice vertex algebras in terms of extension correspondences on Nakajima quiver varieties was accomplished in \cite{Nagao_2007}. In \cite{DeHority_2020} this was extended to certain middle dimensional correspondences between smooth moduli spaces of sheaves on K3 surfaces $S$ providing a representation of the affinization of a Lorentzian Kac-Moody algebra $\widehat{g}_{\NS(S)}$ on the cohomology of some smooth moduli spaces of Bridgeland stable objects on K3 surfaces of the form
\[ \widehat{\mf g_{\NS(S)}} \to \End(\bigoplus_{\alpha \in \widehat{R}_{\NS(S)}} H^*(M_{\sigma}(v + \alpha)))\] 
where $\NS(S)$ is the Neron-Severi lattice of $S$ required to satisfy a natural condition, $\widehat{R}_{\NS(S)}$ is the root system of 
\[ \widehat{\mf g}_{\NS(S)} = \mf g_{\NS(S)}[t^\pm] \oplus \Q c \] 
which is identified with a sublattice of $K_{num}(S)$. 

In particular, by restricting to a family of open analytic subsets of these moduli spaces $M_\sigma(v)$ which are diffeomorphic to the moduli space $M(v; \mc F)$ the generators $x_\alpha$ of $\widehat{\widehat{\mf g_R}}$ in the root space $\widehat{\widehat{\mf g_{R, \alpha}}}$ can be shown to act via correspondences 
\[ \Phi_{C, C'} \circ Z_x \circ \Phi_{C', C} \subset M(v; \mc F)\times M(v+\alpha; \mc F)\]
between moduli spaces of framed sheaves on $\overline{X_R}$ where $\Phi_{C,C'}$ is a cycle giving the isomorphism on cohomology induced by a birational identification between moduli spaces for $\sigma \in C$ and $\sigma' \in C'$ where $C$ and $C'$ are two chambers in $\Pic(M(v; \mc F))$ and $C'$ is adjacent to a wall labelled by $\alpha$ and $Z_x$ is a correspondence induced by the action of a cohomological Hall algebra which locally coincides with the correspondences of Nakajima \cite{Nakajima_1998}. 

The answer to the geometric identification question in this equivariant and non-middle-dimensional setting, which is necessary for applications in enumerative geometry, is provided in the $A_{-1}, D_4$ and $E_{6}, E_7$ and $E_8$ cases by a deformation of the moduli spaces $M(v; \mc F)$.  These moduli spaces $M(v; \mc F)$ like the Hilbert scheme of points, admit proper Lagrangian fibrations to $\A^n(v)$. While we expect the same results to hold in the more general setting, the proofs rely on the fact that in these cases $X_R$ is a resolution of the coarse moduli space of a local elliptic orbifold $T^*[E/\Gamma]$ for $\Gamma\in \Aut(E)$. In these cases, moduli spaces of $\Gamma$-clusters of $T^*E$ are known to coincide with moduli spaces of tamely ramified parabolic Higgs bundles \cite{Groechenig_2014} while the other cases correspond to wild ramification \cite{Boalch_2018}. 

In \cite{dehority_toroidal_springer} we introduce a family of abelian categories $\mathscr A/B_R$ where 
\[ B_R = HH_0([E/\Gamma]) = \Ext^1_{[E/\Gamma]\times [E/\Gamma]}(\Delta_*\mc T_{[E/\Gamma]}, \mc O_\Delta) \]  is an affine base scheme and $\mathscr A = \Coh(\PP_{[E/\Gamma]_{B_R}}(\Xi))$ is a non-commutative $\PP^1$ bundle in the sense of Van den Bergh \cite{vdBergh_2012}. This induces a family of $dg$-categories $D_{\Coh}(\PP_{[E/\Gamma]_{B_R}}(\Xi))$ whose objects are complexes of objects in $\mathscr A$ whose cohomology live in a subcategory corresponding to coherent sheaves. Using the study of moduli spaces for families of noncommutative surfaces \cite{rains19birational_arxiv} and the construction of stability conditions in families of algebraic surfaces \cite{Bayer_Lahoz_Macrì_Nuer_Perry_Stellari_2021} we produce deformations of moduli spaces relevant to our situation.

\begin{thm}[\cite{dehority_toroidal_springer}]\label{thm:main_birational}
  \begin{itemize}
    \item[1)] If $\Gamma = \{ 1\}$ or subject to conjectures otherwise on $D_{\Coh}(\PP_{[E/\Gamma]}(\Xi))$ in general, 
there is a nonempty open subset 
\[U \subset \Stab(D_{\Coh}(\PP_{[E/\Gamma]_{B_R}}(\Xi)))\] and relative moduli spaces 
\[ M_{\underline \sigma}(v; \mc F)/HH_0([E/\Gamma]) \]
which are quasiprojective over $HH_0([E/\Gamma])$, semiprojective, and birational for any $\underline\sigma, \underline \sigma' \in U$. 
\item[2)] There is a 1-1 correspondence between walls in $N^1(M(v ;\mc F)/\A^{n(v)})_\R$ corresponding to singular birational models, walls in $HH_0([E/\Gamma])$ where the contraction is not biregular, and a specific subset of roots in $R^{ell}$. 
  \end{itemize}
\end{thm} 

The above theorem provides a deformation of the restriction of the results of \cite{Bayer_Macrì_2014_MMP} to the neighborhood of certain singular fibers in certain isotrivial Lagrangian fibrations of $K3^{[n]}$ type of the form $S^{[n]}\to \PP^n$ induced by analytic inclusions $X_R\hookrightarrow S$ where $S$ is an isotrivial K3 surface which induce $X_R^{[n]} \to S^{[n]}$.

In the special case of $\Gamma = \{ 1\}$ and when $\mc F \simeq \mc O_{D_\infty}$, the space $B_R \simeq K_{num}(E)$ and has a $1$-dimensional subspace $B_{CM} \subset B_R$ such that for any $b \in B_{CM}$ the space $M_{\underline{\sigma}}(v; \mc F)$ is isomorphic to the elliptic Calogero-Moser space. In this same case, the fibers of the map $M_{\underline{\sigma}}(v; \mc O_{D_\infty})$ for special choices of $v$ and $\underline{\sigma}$ are good moduli spaces for the stacks studied in \cite{ben2008perverse}. 

The quasiprojectivity of the moduli space $M_{\underline\sigma}(v; \mc F)$ should be compared to the situation for deformations of varieties of $K3^{[n]}$-type where the generic deformations are non-projective and the total spaces of twistor lines fail to be K\"ahler. In this situation the identification between the sets of walls should be seen as an analogue of the main result of \cite{Namikawa_2015} which identifies walls in the space of Poisson deformations of a resolution $X$ of an  affine conical symplectic singularity $Y$ with walls in the real Picard group of $X$. When $X$ is a Nakajima quiver variety, these walls are also in bijection with a set of roots for the corresponding Kac-Moody algebra.  The deformations provided by Theorem \ref{thm:main_birational} are closely related to those of the nonabelian Hodge correspondence \cite{simpson1990nonabelian}. 

Recall that for a Poisson variety $(X, \{ -,-\})$ the space of deformations of $(X, \{ -,-\})$ is controlled by its Poisson cohomology group $HP^2((X, \{ -,-\}))$ which inherits an action of a group $\C^\ast_\hbar$ on $X$ that scales $ \{ -,-\}$ with weight $-\ell < 0$. We expect that 
\[ HH_0([E/\Gamma]) \simeq HP^2_+((M_{\underline{\sigma}}(v; \mc F), \{ -,-\}))\] 
where $HP^2_+(-)$ is the direct sum of positive weight spaces of $HP^2(-)$. The space $HH_0([E/\Gamma])$ coincides with the positive weight part of the tangent space to $\PP(\mc T_{[E/\Gamma]} \oplus \mc O_{[E/\Gamma]})$ in the moduli stack of noncommutative $\PP^1$-bundles over $[E/\Gamma]$. 

Theorem \ref{thm:main_birational} is based on a local version of the Bayer-Macr\`i map in families, which in our case is a map  $\ell : U \to N^1(M(v ;\mc F)/\A^{n(v)})_\R$. The subspace of stability conditions $U$ is an analogue in the derived category of the Friedman chamber $C \subset N^1(S)$ of polarizations $H$ on a projective elliptic surface $S\to C$ in which moduli spaces of Gieseker semistable sheaves are well behaved. Theorem \ref{thm:main_birational} is based on an analysis of the action of derived autoequivalences of elliptic surfaces on the stability manifold in \cite{Lo_Martinez_2022,Liu_Lo_Martinez_2019,Yoshioka_2022}. 

\subsection{Category of modules for $\mc A_\hbar(\mf g_{X_R})$} 

We expect that only for $R$ in the exceptional series \eqref{eq:exceptional_series} and perhaps also the non-simply-laced $R = G_2$ or $F_4$, does there exist an analogue with geometric origin of the algebras 
\[ Y_{\hbar}(\mf g_{KM}), ~~~ U_q(\widehat{\mf g_{KM}})\] 
which are the Yangian and quantum affine algebra, Hopf algebras which are filtered deformations over $\C[\hbar]$ or $\C[q^\pm]$ respectively of $\on{Maps}(C, \mf g_{KM})$, maps into a Kac-Moody algebra from $C = \A^1$ in the Yangian case or $C = \C^\ast$ in the quantum affine case. We specialize here to the case $R = A_{-1}$ and provide a definition of the algebra $\mc A_\hbar(\mf g_{T^*E})$ which is conjecturally a Hopf algebra and as previously discussed meant to be an analogue of the Yangian $Y_\hbar(\mf g_{KM})$. 

Let $\mc F_1$ and $\mc F_2$ be stable framing sheaves of positive rank with an inequality $\mu(\mc F_1) \le \mu(\mc F_2)$ of slopes. 
Let $M^{tf}(v; \mc F_1 \oplus \mc F_2)$ denote the stack of $\mc F_1 \oplus \mc F_2$-framed torsion-free sheaves on $\overline{X_{A_{-1}}}$. Owing to the fact that $\overline{X_{A_{-1}}} = E\times \PP^1$ is a ruled surface, a canonical filtration on torsion-free sheaves analogous to the Brosius filtration controls the stabilizers of points in this stack. 

\begin{thm}\label{thm:main_tensor}
There is an isomorphism 
\[\bigoplus_{m,n\in \Z} H^*_T(M^{tf}(v + m[\pt] + n[E]; \mc F_1 \oplus \mc F_2))_{loc} \simeq V_{T^*E, \mc F_1} \otimes V_{T^*E, \mc F_2}\] after localization of the base ring $H^*_T(\pt)$. 
\end{thm}

Owing to the Hopf algebra structure on $U(\mf g_{T^*E})$ we obtain a representation of $\mf g_{T^*E}$ on the vector space of $\bigoplus_{m,n\in \Z} H^*_T(M^{tf}(v; \mc F_1 \oplus \mc F_2))_{loc}$. Furthermore there exist well defined determinant line bundles $\lambda(F) \in N^1_{\C^\ast_\hbar}(M^{tf}(v; \mc F_1 \oplus \mc F_2))$ given $F \in K_{\C^\ast_\hbar}(\overline{X_{A_{-1}}})$. More generally this construction provides a tensor decomposition of the cohomology of moduli spaces for any framing bundle of the form $\mc F = \oplus_{i\in I} \mc F_i$ whose summands are pairwise non-isomorphic. The corresponding module depends only on the chern classes $A = \{ [\mc F_i]\}$ of the $\mc F_i$ and we define 
\[ V_A := \bigoplus_{m,n\in \Z} H^*_T(M^{tf}(v + m[\pt] + n[E]; \mc F))\] 
and finally define 
\[ \mc A_\hbar(\mf g_{T^*E}) \subset \prod_A \End(V_A) \] 
to be the algebra generated by cup product by all determinant line bundles $\lambda(F)$ and by the action of $\mf g_{T^*E}$. Owing to formulas for the cup product operators and Theorem \ref{thm:main_nonlog_thm} this gives a computationally effective definition of $\mc A_\hbar(\mf g_{T^*E})$ which can be determined on graded pieces according to the grading on the tensor product induced by the weight grading on all factors rather than the weight grading on the module $V_{A}$, which has infinite dimensional weight spaces once $|I| > 1$. 

There is an analogue of Theorem \ref{thm:main_tensor} over a base $B$ which is the complement of the diagonal in $\Jac(E) \times \Jac(E)$ supposing $[\mc F_1] = [\mc F_2]$ and equals $\Jac(E) \times \Jac(E)$ otherwise. There are analogous statements for the other surfaces $X_R$. 

Towards the study of these tensor products and the coproduct structure, in Theorem \ref{thm:higher_rank} we obtain a representation of higher dimensional toroidal extended affine Lie algebras on the tensor product of $n$ copies of $V_{T^*E}^{\Jac{}}$. 

\subsection{Quantum differential equation } 

The image of the Bayer-Macr\`i map is a half space 
\[ N^1(M(v ;\mc F)/\A^{n(v)})_{\R, >0} \subset N^1(M(v ;\mc F)/\A^{n(v)})_\R  \] 
consisting of those $ D\in N^1(M(v ;\mc F)/\A^{n(v)})_\R$
 such that $(D, C_E) > 0$ where $C_E$ is the class of a curve in one of the smooth fibers of the fibration 
 $M(v ;\mc F)\to A^{n(v)}$. 

The space $N^1(M(v ;\mc F)/\A^{n(v)})_\R$ and the space of stability conditions $U$ are closely related to the K\"ahler moduli space 
\[ \mc K_R  \subset N^1(M(v ;\mc F)/\A^{n(v)})\otimes \C^\ast  \] 

of $M(v ;\mc F)$. There are coordinates of $\mc K_R$ around cusps $\infty_C \in \overline{\mc K_R }$ which lie away from the hyperplane corresponding to the natural half space boundary 
$\on{Re}(\log(z)), C_E) = 0$ of  $N^1(M(v ;\mc F)/\A^{n(v)})_{\R, >0}$. There is a bijection between the set of of cusps and chambers of the wall-and-chamber structure of Theorem \eqref{thm:main_birational}. Expansion around the cusp $\infty_C$ give formal coordinates for the (reduced, equivariant, modified) quantum cohomology ring 
\[ \gamma_1 \ast_z \gamma_2 = \sum_{\beta \in \NE((M(v ;\mc F)/\A^{n(v)}))} z^\beta(\gamma_1 \ast_z \gamma_2)_\beta \in H^*_T(M(v ;\mc F))[[\NE(M(v ;\mc F)/\A^{n(v)})]]\] 
of $M_\sigma(v; \mc F)$ the birational model corresponding to the chamber $C$. 

The main result of \cite{dehority_quatum_connection} is an identification of the quantum multiplication of $M(v; \mc F)$ using the algebra $\mc A_\hbar(\mf g_{X_R})$. This gives an analogue of the main theorem of \cite{Maulik_Okounkov_2019} for these moduli spaces. 

Implicit in the wall-and-chamber structure of \ref{thm:main_birational} is an identification of the lattice of curves $N_1(M(v ;\mc F)/\A^{n(v)})$ with the root lattice $Q^{ell}$ of the elliptic root system. The root spaces of $\mf g_{X_R}$ are in bijection with roots of $R^{ell}$ and the canonical pairing between root spaces $\mf g_{\beta}$ and $\mf g_{-\beta}$ gives an element $e_{-\beta}e_{\beta}$ with an implicit sum over a basis of $\mf g_{\beta}$. 

In the general case, all curve classes except for those which are numerically equivalent to those living in a smooth fiber of the lagrangian fibration have their contributions to quantum multiplication given by a multiple of the operator $e_{\beta}e_{-\beta}$. 

\begin{thm}[\cite{dehority_quatum_connection}]\label{thm:main_QDE}
Up to a scalar operator, quantum multiplication by divisors in $M(v; \mc F)$ is given by
\begin{equation}
  \label{eq:main_qde_ze_abstract} c_1(\lambda) \ast_z - = c_1(\lambda) \cup (-)  - \hbar \sum_{\beta \in R^{ell}_{>0}\backslash \Z\delta_E  } (\lambda, \beta) \frac{z^\beta}{1 - z^\beta} e_{\beta}e_{-\beta} + \Xi(z^{\delta_E}).\end{equation}
\end{thm}

The series defining the matrix coefficients of the operators \eqref{eq:main_qde_ze_abstract} converge in a region of $\mc K_R$ and analytically continue to all of $\mc K_R$. 

The general theory of quantum multiplication implies that quantum multiplication by divisors give rise to a connection on a trivial bundle over $\mc K_R$ with fiber $H^*_T(M(v; \mc F))$ defined by 
\[ \nabla_{\lambda} = \frac{d}{d\lambda}  - c_1(\lambda) \ast_z (-) \] 

In the special case of $R = A_{-1}$ we will use in \cite{dehority_quatum_connection} the vertex operator algebra methods of the present paper together with structural properties of monodromy invariance of quantum multiplication and the Witten–Dijkgraaf
–Verlinde–Verlinde (WDVV) equation 
\[ [\nabla_\lambda, \nabla_\mu ] = 0 \]
 to uniquely determine that 
\begin{equation}\label{eq:qde_E_corrections} \Xi(z^{\delta_E}) = -\hbar \sum_{m\in \Z} (\lambda,m \delta_E) \frac{z^{m\delta_E}}{1- z^{m\delta_E}} e_{m\delta_E}e_{-m\delta_E}\end{equation}
up to a scalar function. We expect the corresponding result to also hold outside of type $A_{-1}$. 

\subsubsection{Acknowledgements} The present project owes a tremendous amount to a great many people. Crucial to its completion were ideas of O. Schiffmann and also my advisor A. Okounkov.

\section{Toroidal algebras and superalgebras}

\subsubsection{Superalgebras and supergeometry} 

Given an even vector space $V$ of dimension $m$ we let $\C[\Pi V] = \C[\theta_1, \ldots, \theta_m] \simeq \Lambda^*(V^\vee) \simeq H^0(\mc O_{\A^{0|m}})$ denote its corresponding odd ring of functions. Let $\theta_I = \theta_{i_1} \cdots \theta_{i_k}$ for $I \subset \{ 1, \ldots, m\}$.  Given a function $f(\theta) \in H^0(\mc O_{\A^{0|m}})$ with expansion 
define $\int d\theta : \C[\Pi V] \to \C$ by $\int d\theta \theta_{\{1, \ldots,n\}} = 1$ and extend by linearity. Letting $\psi_1, \ldots, \psi_m$ be generators for $\C[\Pi V^\vee]$ we define the Fourier transform of $f \in \C[\Pi V]$ by 
\begin{equation}
  \label{eq:fourier} \widehat {f}(\psi) = \int d\theta e^{\sqrt{-1} \psi_i \theta_i} f(\theta).
\end{equation}

A Lie superalgebra $\mf g = \mf g_{\underline 0} \oplus \mf g_{\underline 1}$ is always assumed to have a purely even bracket \[ [ -,-]: \mf g_{\underline i}\otimes \mf g_{\underline j} \to \mf g_{\underline{ i + j}}.\] 

A (supercommutative) Poisson superalgebra is a pair $(A, \{ -,-\})$ consisting of a commutative associative superalgebra $A = A_{\underline 0} \oplus A_{\underline 1}$ with a Poisson bracket 
\[ \{ -, - \} : A\otimes A \to A \] 
giving a Lie superalgebra structure satisfying the Leibniz rule 
\[ \{ a, bc\} = \{ a,b\}c + (-1)^{p(a)p(b)}b \{ a,c\}. \] 

if $\mf g = H^0(\mc O_{\A^{0|m}}) = \C[\xi_1, \ldots, \xi_m]$ is a free supercommutative ring on $m$ odd generators then there is a Poisson bracket 
\[ \{ f,g\} = (-1)^{p(f) + 1} \sum_i \pd{f}{\xi_i}\pd{g}{\xi_i}. \] When $m = 2n$ we choose coordinates $\psi_{+,1},\psi_{-,1}, \ldots,  \psi_{+,n},\psi_{-,n}$ so that the bracket has the form 
\[ \{ f,g\} = (-1)^{p(f) + 1} \sum_i \pd{f}{\psi_{+,i}}\pd{g}{\psi_{-,i}}+ \pd{f}{\psi_{-,i}}\pd{g}{\psi_{+,i}}. \]

The Lie algebra structure denoted $\mf{po}(0|m)$ and known as the \emph{Poisson Lie superalgebra}. It has center spanned by $1 \in H^0(\mc O_{\A^{0|m}})$ and the quotient denoted $\mf h(0|m)$ is the \emph{Hamiltonian Lie superalgebra}. The pairing  $\langle -, -\rangle :  \mf g\otimes \mf g\to \C$ defined by 
\begin{equation}
  \langle f,g\rangle  = \int d\xi fg \label{eq:poisson_super_pairing} 
\end{equation}
is a non-degenerate supersymmetric invariant form on $\mf{po}(0|m)$.

Let $\C^{\ast a | b}$ denote an affine superscheme \cite{Carmeli_Caston_Fioresi_2011} whose underlying bosonic scheme is the algebraic torus $\C^{\ast a}$ and whose odd part is the parity shift of a vector bundle over $\C^{\ast a}$. Because all projective modules over $\C^{\ast a}$ are trivial we know that as a superscheme $\C^{\ast a | b}$ is a product $\C^{\ast a |0 }\times \A^{0|b}$. 
 
Now specialize to $ a= b = 2$. For $\alpha = (\alpha_{ev}, \alpha_{odd}) \in \C^2$ consider the Poisson bracket on $H^0(\mc O_{\C^{\ast a | b}})$ defined by 
\begin{equation}\label{eq:ev_plus_odd_pb}
\{ - , -\}_\alpha = \alpha_{ev} \{ -,-\}_{ev} +  \alpha_{odd}\{ -,-\}_{odd}
\end{equation}
where 
\[ \{f, g \}_{ev} = \pd{f}{\log s}\pd{g}{\log t} - \pd{f}{\log t} \pd{g}{\log s}\] is the even Poisson bracket in logarithmic coordinates and 
\[  \{ f, g\}_{odd} = (-1)^{p(f) + 1} \sum_i \pd{f}{\psi_{+,i}}\pd{g}{\psi_{-,i}}+ \pd{f}{\psi_{-,i}}\pd{g}{\psi_{+,i}}
  \] is induced from the bracket on $\mf{po}(0|2)$. 

  Let   
  \begin{equation}\label{eq:ham} \on{Ham}(\C^{\ast 2 | 2})_{\alpha} = \Q \langle \{ f, - \}_\alpha \rangle \subset \on{Der}(H^0(\mc O_{\C^{\ast 2 | 2}})) 
  \end{equation}
  denote the Lie superalgebra of Hamiltonian vector fields with respect to the Poisson bracket $\{ -,-\}_{\alpha}$. 

  \subsection{Toroidal algebras} 
\subsubsection{} 

Let $\mf g$ be a semisimple finite dimensional Lie algebra which for us will be allowed to be $\mf{gl}_0 = \on{Lie}(\{ 1 \} )$ or a finite dimensional Lie superalgebra with a supersymmetric invariant form. Let $A = H^0(\mc O_{\C^{\ast 2}_{s,t}}) = \C[s^{\pm}, t^{\pm}]$ be the ring of functions on an algebraic 2-torus $ \C^{\ast 2}_{s,t}$. Consider the space of global 1-forms $\Omega_{A} = H^0(\Omega_{\C^{\ast 2}_{s,t}}) =  A ds \oplus A dt$ and global vector fields $\mc T_{A} =  A \mrm d_s \oplus A \mrm d_t$ where $\mrm d_s = s\pd{}{s}$ and $\mrm d_t = t\pd{}{t}$. Let $\mrm k_s = s^{-1}ds$ and $\mrm k_t = t^{-1}dt$ be a distinguished basis for $\Omega^1_A$. Let $\mc K = \Omega^1_{A}/ d\Omega^0_{A}$ and $\mc D_{div} = \ker(\on{div}) \subset \mc T_{A}$ denote the quotient of 1-forms modulo exact forms and the subspace of divergence free vector fields respectively. Here $\on{div}(f_s \mrm{d}_s + f_t\mrm{d}_t) = s\pd{f_s}{ds} + t\pd{f_t}{dt}$ is proportional to the divergence in logarithmic coordinates.   The full toroidal Lie algebra, due to Kassel \cite{Kassel_1984} is the Lie algebra 
\[ \mc T(\mf g) = \mf g\otimes A \oplus \mc K \oplus \mc T_{A}.\] 
Its subalgebra 
\[ \mf t(\mf g) = \mf g\otimes A \oplus \mc K \oplus \mc D_{div}\] 
is a toroidal extended affine Lie algebra. 
Let $V_{a,b} = bx^ay^bd_x - ax^ay^bd_y$. 

These algebras and their relationship to vertex algebras have been extensively studied \cite{Moody_Rao_Yokonuma_1990,Billig_2006,Billig_2007}. We will use the basis described in \cite[\S 2]{Chen_Li_Tan_2021} on account of its relationship with \eqref{eq:main_bracket}. 

For $(a,b) \in \Z^2 \backslash \{(0,0)\}$ and $x\in \mf g$ let 
\begin{align}
  \label{eq:kab}
   \mrm k_{a,b} &= \begin{cases} \frac 1 b s^at^b\mrm k_s & \text{ if } b \neq 0  \\
    -\frac{1}{a} s^a \mrm k_t & \text{ if } b = 0, a \neq 0 
   \end{cases}\\
   \label{eq:dab}
   \mrm d_{a,b} &= as^at^b\mrm d_a - bs^at^b \mrm d_b\\
   x_{a,b} &= s^at^b \otimes x \label{eq:xab}
  \end{align}
We have the following commutation relations 
\begin{align}
[x_{a,b}, y_{c,d}]  &= [x,y]_{a+c,b+d} - \langle x, y\rangle (ad-bc)\mrm k_{a + b, c+d} + \delta_{a + c, 0} \delta_{b + d, 0} (a\mrm k_s + b \mrm k_t)\\
[\mrm d_{a,b}, x_{c,d}] &= -(ad-bc)x_{a+c, b+d} \\
[\mrm d_{a,b}, \mrm k_{c,d}] &= -(ad - bc)\mrm k_{a + c, b+  d} + \delta_{a + c, 0} \delta_{b + d, 0}(a \mrm k_s + b \mrm k_t)\\
[\mrm d_{a,b}, \mrm d_{c,d}] &= -(ad - bc)\mrm d_{a + c, b+  d}.
\end{align}

Define the following generating functions of elements of $\mc T(\mf g)$: 

\begin{align}
J_{x,a}(z) &= \sum_{n \in \Z}s^at^nz^{-n-1} \\
 \mrm{d}_{s,a}(z) &=  \sum_{n \in \Z} s^at^n\mrm{d}_s z^{-n-1}\\
 \mrm{d}_{t,a}(z) &=  \sum_{n \in \Z} -s^at^n\mrm{d}_t z^{-n-1}\\
 \mrm{k}_{a}(z) &=  \begin{cases}\sum_{n \in \Z} s^at^n\mrm{k}_t z^{-n-1} & \text { if } a \neq 0 \\
  \sum_{n \in \Z} t^n\mrm{k}_s z^{-n-1} & \text { if } a = 0 \end{cases}.
\end{align}

\subsubsection{Elliptic root systems}

An elliptic root system in a rank $\ell$ semidefinite real or complex inner product space $(F, \langle -, -\rangle)$ of dimension $\ell + 2$ is defined by a distinguished set of real roots $R_{re}\subset F$ and imaginary roots $R_{im} \subset \on{rad}(\langle -, -\rangle)$ whose union $R = R_{re} \cup R_{im}$ is the root system. A marking of $R$ is a choice of a line $G \subset \on{rad}(\langle -, -\rangle)$. This choice induces an affine root system $R_{aff} = R/(R\cap G)$ on the semidefinite space $F/G$ and from this, as usual, a finite root system $R_{fin} = R/R \cap \on{rad}(\langle -, -\rangle)$ on $F/\on{rad}(\langle -, -\rangle)$.  The pair $(R,G)$ is called in \cite{Saito_Takebayashi_1997} a marked elliptic root system. 
\begin{defn}
  A \emph{fully marked elliptic root system} is a triple $(R, \delta_1, \delta_2)$ where $R$ is an elliptic root system and $\delta_1$ and $\delta_2$ form a basis of $R_{im}$. 
\end{defn}

The definition of an elliptic root system differs slightly from the one in \cite{Saito_1985} in that the root system consists of the real and imaginary roots and the span of the real roots is not required to be the entirety of $F$. The main example of an elliptic root system according to the definition here but not in \cite{Saito_1985} is the one where $\ell = 0$ and $R_{fin}$ has Cartan type $A_{-1}$ corresponding to the empty root system on the zero dimensional positive definite vector space $\Bbbk^0$. Then $F \simeq \Bbbk^2$ and $R\simeq \Z^2$ is a lattice in $F$ of imaginary roots. 

Given a finite dimensional root system 
let 
\[ R^{ell} = \{ \beta + m \delta_1 + n\delta_2 \mid \beta \in R_{fin} \sqcup\{ 0 \} ,~  m,n \in \Z \} \subset F_{fin} \oplus \Bbbk \delta_1 \oplus \Bbbk \delta_2 \] 
denote the corresponding elliptic root system. 

\subsubsection{Elliptic Weyl groups} 

Given an elliptic root system $R^{ell} \subset F$ the Weyl group of $R^{ell}$ is the subgroup 
\[ W_R^{ell} \subset \on{Aut}(F, \langle - , -\rangle )\]  generated by reflections $w_{\beta}$ through $\beta \in R^{ell}_{re}$. 

The choice of marking $G$ determines a splitting 
\begin{equation}\label{eq:elliptic_weyl_split} 0 \to  \Z[R_{fin}] \to W_R^{ell} \to W_{R}^{aff} \to 0 \end{equation}
 via the induced action on $R_{aff}$ where $W_R^{aff}$ is the affine Weyl group associated to $R$. Recall also that $W_{R}^{aff}$ itself fits into a similar exact sequence
\begin{equation}\label{eq:affine_weyl_split} 0 \to \Z[R_{fin}] \to W_{R}^{aff} \to W_R \to 0. \end{equation}

% \subsubsection{Affinized star-shaped root systems} 

% More generally, consider a root system $R_{KM} \subset \mf h$ where $\mf h \subset \mf g_{KM}$ is a Cartan subalgebra on a Kac-Moody algebra with star-shaped Dynkin diagram. An affinized start-shaped root system consists of the set 
% \[ R^{aff}_{KM} = \{ \beta + m \delta \mid \beta \in R_{KM} \sqcup \{ 0 \}, ~m\in \Z \} \subset \mf h \oplus \Bbbk \delta \] 
% where the pairing on $\mf h \oplus \Bbbk \delta $  is induced from that on $\mf h$. The Weyl group  $W_{R_{KM}}^{aff} \subset \Aut(\mf h \oplus \Bbbk \delta )$ is the set of reflections through real roots and it likewise fits into an exact sequence 
% \[ 0 \to \Z[R_{KM}] \to W_{R_{KM}}^{aff} \to W_{R_{KM}} \to 0 .\] 

% Let $Q_R$ denote the root lattice associated to $R$, $\widehat Q_R$ denote the corresponding affine root lattice and let $Q^{ell}_R = Q_R \oplus K_{num}(E)$ denote the corresponding root lattice. 

\subsubsection{Imaginary reflections, Fourier transforms and Weyl groups} 

By definition for an elliptic root system $R^{ell}$ the elliptic Weyl group $W_{R}^{ell}$ preserves $\on{rad}(\langle-,-\rangle)$ pointwise. Geometrically at least in the elliptic orbifold cases corresponding to elliptic root systems such that 
\[ K_0([E/\Gamma]) \simeq Q_R^{ell}\]
when we have an action 
\[ W_R^{ell} \to \GL(K_0([E/\Gamma])) \] 
the elliptic Weyl group will only be a subgroup of those automorphisms induced by families of derived equivalences and there will be outer automorphisms of the corresponding toroidal and toroidal extended affine Lie algebras not arising from the Weyl group. 

Because of this in the elliptic orbifold or $A_{-1}$ cases we let 
\begin{equation}\label{eq:GW_def} 
IW_R^{ell} \subset GL(Q_R^{ell}\otimes \C)
\end{equation} 
denote the group generated by $W_R^{ell}$ and $\Aut(D_{Coh}([E/\Gamma]))$, where by convention the latter contains anti-autoequivalences. Because the group of autoequivalences surjects onto $\Aut(\on{rad}(\chi_{[E/\Gamma]}(-,-))) \simeq \GL(2, \Z)$ \cite{Burban_Schiffmann_2013} we have an exact sequence 
\[ 0 \to W_R^{ell} \to IW_R^{ell}\to \GL(2, \Z) \to 0.  \] 

when the elliptic root system is seen as the root system of a toroidal algebra $\mf g_{tor} \supset \Hom(\C^{\ast 2}, \mf g)$ the imaginary reflections correspond to automorphisms of the underlying torus $\C^{\ast 2}$. 

In particular, there are derived equivalences in $\Aut(D_{Coh}([E/\Gamma]))$ whose induced adjoint action on $GW^{ell}_R$ exchange the two left hand terms isomorphic to $\Z[R_{fin}]$ in \eqref{eq:elliptic_weyl_split} and \eqref{eq:affine_weyl_split} which are both normal subgroups of $W_R^{ell}$. 

A reflection group lying between $IW_R^{ell}$ and $W_R^{ell}$ which preserves vertex operator algebra representations of toroidal algebras is $W_{R,G}^{ell}$ which depends on a marking $G$. Define 
\begin{equation}\label{eq:monodromy_weyl_ell} W_{R,G}^{ell} := \{ w \in IW_R^{ell} \mid w G = G \}  \end{equation}
so that since the stabilizer of a line in a lattice is $\Z \rtimes \Z/2\Z \subset \GL(2, \Z)$ we have an exact sequence 
\[ 0 \to W_R^{ell} \to W_{R,G}^{ell} \to \Z\rtimes \Z/2\Z \to 0 \] 
and letting $W_R^{aff, \pm}\subset \Aut(\widehat{\mf h})$ denote the group generated by $W_R^{aff, \pm}$ and the reflection through the imaginary root hyperplane, we have an exact sequence 
\begin{equation}\label{eq:affine_translation_ell_seq}
0 \to \Z[R_{aff}] \to W_{R,G}^{ell} \to W_R^{aff, \pm} \to 0 
\end{equation}
analogous to \eqref{eq:elliptic_weyl_split}. When $R = A_{-1}$ we have $W_{A_{-1}}^{ell} = \{ 1 \}, IW_{A_{-1}}^{ell} = \GL(2, \Z)$ and $W^{ell}_{A_{-1},G} = \Z \rtimes \Z/2\Z$.

\subsubsection{Lie algebra in case $R = A_{-1}$}\label{ssec:gamin1}
In this section we introduce the main Lie algebra under consideration in this paper. 
Let $\mf g_{T^*E} = \bigoplus_{r,d} \mf g_{r,d}\oplus \Q \pmb{c}$ where $\mf g_{r,d}\simeq H^*(E)$.  This forms a Lie super algebra with bracket 
\begin{equation}\label{eq:main_bracket_non_intro} [w^{a,b}_{\gamma}, w^{c,d}_{\gamma'}] = -\det\begin{pmatrix}a & b \\ c & d \end{pmatrix} w^{a + c, b + d}_{\gamma \star \gamma'}  + \delta_{a+c, 0} \delta_{b + d, 0}\langle \gamma, \gamma'\rangle (a\pmb{c}_s + b\pmb{c}_t).
\end{equation}
\begin{prop}  \label{prop:lie}
The bracket \eqref{eq:main_bracket_non_intro} determines a Lie superalgebra structure. 
\end{prop}

\begin{proof}
  The algebra of convolution on $H^*(E)$ is a free supercommutative ring on two generators. 
The bracket \eqref{eq:main_bracket_non_intro} modulo central extension on $\mf g/ (\Q \pmb c_s \oplus \Q \pmb c_t)$ may be identified with $\on{Ham}(\C^{\ast 2 | 2})_{(1,0)}$ corresponding to the purely even Poisson bracket in \eqref{eq:ham}.  The relevant central extension comes from a cocycle $\xi : \mf g/ (\Q \pmb c_s \oplus \Q \pmb c_t) \otimes \mf g/ (\Q \pmb c_s \oplus \Q \pmb c_t) \to \Q^2 $. 
\end{proof}

Let $\mc Heis_{H^*(E), k}$ denote the Heisenberg-Clifford algebra modelled on $H^*(E)$ with central charge $k$ and $\mc F_{H^*(E), k}$ its vacuum module.  Denote $\mf g_0 = \mf g/ (\Q \pmb c_s \oplus \Q \pmb c_t)$. Let $B$ be a basis for $H^*(E)$. 

Given a slope $\mu = a/b  \in \Q\sqcup\{\infty \}$ we have a slope subalgebra 
\[\mf g_{\mu} = \Q \langle w_{\gamma}^{ka, kb}, \pmb c_\mu \rangle ~~ k \in \Z\backslash 0, \gamma \in B \] 
where $\pmb c_\mu = a \pmb c_s + b\pmb c_t$ where $\mu = a/b$ is written in reduced terms with $b \ge 0$. 

Any slope subalgebra $\mf g_\mu$ is isomorphic to $\mc Heis_{H^*(E)}$. The Lie algebra $\mf g_0$ admits an automorphism group of $\SL(2, \Z)$ acting on $\C^{\ast 2 | 2}$ acting on the even coordinates because this group preserves the Poisson form. This action acts via the vector representation of $\SL(2,\Z)$ on the image of the $\Q^2$-valued cocycle $\xi = (\xi_s, \xi_t)$ from the proof of Proposition \ref{prop:lie}. Given any linear functional $\Q\langle \pmb c_s, \pmb c_t\rangle  \to \Q$ corresponding to the action of $\mf g$ in an irreducible representation there is a slope subalgebra $\mf g_\mu$ on which the central extension $\pmb c_\mu$ vanishes and hence $\mf g_{\mu}$ is supercommutative. 

The reader may verify that the cocycle $\xi$ from the proof of Proposition \eqref{prop:lie} also provides a central extension of $\on{Ham}(\C^{\ast 2 | 2})_\alpha$ for any $\alpha$. Denote this central extension 
\[  \widehat{\on{Ham}}(\C^{\ast 2 | 2})_\alpha = \on{Ham}(\C^{\ast 2 | 2})_\alpha\oplus \Q \pmb c_s \oplus \Q \pmb c_t. \] 

For $\alpha_{odd} \neq 0$ the slope subalgebra $\mf g_\mu\subset  \widehat{\on{Ham}}(\C^{\ast 2 | 2})_\alpha$ defined in the same way are copies of the affine Lie superalgebra $\widehat{\mf{po}(0|2)}$.

\subsubsection{Jacobian toroidal Lie algebra}
\label{ssec:jac_g}

The Lie algebra $\mf g_{T^*E}$ after base change to $H^*_T(\Jac(E))$ admits an extension to $\mf g_{T^*E}^{\Jac}$. 

Let $H^*(\Jac(E)) = \C[\xi_+, \xi_-] \simeq \C[\A^{0|2}]$. Consider odd central elements $w^{0,0}_{\sigma_\pm}$ which acts by multiplication by $\xi_\pm$ respectively. Then 
\begin{equation}
  \label{eq:g_jac}\mf g_{T^*E}^{\Jac} = \mf g_{T^*E}\otimes H^*_T(\Jac(E)) \oplus  H^*_T(\Jac(E)) \big\langle w^{0,0}_{\sigma_+},w^{0,0}_{\sigma_-}, \pd{}{\xi_+}, \pd{}{\xi_-}\big \rangle \end{equation} 
is the Lie algebra generated by $\mf g_{T^*E}$ the new odd zero modes and the odd derivations on the base. 

Its subalgebra $\mf g_{T^*E}^{\Jac{} \tilde{}}\subset \mf g_{T^*E}^{\Jac{}}$
is defined as 
\begin{equation}
  \label{eq:g_jac}\mf g_{T^*E}^{\Jac{}\tilde{}} = \mf g_{T^*E}\otimes H^*_T(\Jac(E)) \oplus  H^*_T(\Jac(E)) \big\langle w^{0,0}_{\sigma_+},w^{0,0}_{\sigma_-}\big \rangle \end{equation} 
omitting the derivations. 

\subsubsection{Full Jacobian algebra }\label{ssec:full_jacobian_lie} 

We will also define in Definition \ref{defn:Ubullet} below the full Jacobian algebra $U^\bullet(\mf g_{T^*E}^{\Jac{}})$ which is of the form 
\[ U^\bullet(\mf g_{T^*E}^{\Jac{}}) = U^\bullet(\mf g_{T^*E}^{\Jac{}})_{\underline 0}\oplus  U^\bullet(\mf g_{T^*E}^{\Jac{}})_{\underline 1}\] 
where the even part $ U^\bullet(\mf g_{T^*E}^{\Jac{}})_{\underline 0}$ contains the universal enveloping algebra of the full toroidal Lie algebra of $\mf {gl}_0$ 
\[ \mc T(\mf{gl}_0) = \Omega^1_{\mc A}/d\Omega^0_{A} \oplus \mc T^1_{ A}\] 
consisting of the central term of the toroidal algebra and the Lie algebra of global derivations on $ A$. We expect that $U^\bullet(\mf g_{T^*E}^{\Jac{}})$ is closely related to a central extension of the subalgebra of 
\[ \mc D (H^0(\mc O_{\C^{\ast 2 | 2}})) \otimes H_T^*(\Jac E) \] 
, the algebra of superdifferential operators on $\C^{\ast 2 | 2}$ generated by the even derivations $\mrm{d}_s, \mrm{d}_t$ and by $\mf g_{T^*E}^{\Jac}$ with the central extension induced by the one on $\mf g_{T^*E}^{\Jac}$.

\section{Surfaces and correspondences}\label{sec:surfaces}

\subsubsection{Correspondences} 

Given $f: X \to Y$ proper, we follow the convention of \cite{Maulik_Okounkov_2019} that a pushforward 
\[ \tau_f(-) : H^*_T(X) \to H^*_T(Y)\] 
is defined by $\tau_f(\gamma) = (-1)^{(\dim X-\dim Y)/2}\int_{X/Y}(\gamma)$ 
where the base ring has a chosen $\sqrt{-1}$. When $f: X \to \pt$ is the map to a point we simply denote the map $\tau(-)$.  We define the intersection pairing $\langle -, -\rangle$ and Poincare duality $\on{PD}$ using this pushforward.

\subsection{Fourier transforms} 

We describe here a relationship between the Fourier-Mukai transform and the Fourier transform for differential operators in odd variables. Let $A$ be an abelian variety and $A^\vee = \Pic^0(A)$ its dual abelian variety over $\Bbbk$. When $\Bbbk = \C$ identify $A = V/\Lambda$ with $V = \C^g$ so that $A^\vee = V^\vee/\Lambda^\vee$. Let $\mc P \in \Pic(A\times A^\vee)$ is the Picard line bundle or universal line bundle. 
The Fourier-Mukai transform 
\begin{align}\label{eq:FM}
\Phi&: D^b(\Coh(A)) \to D^b(\Coh(A^\vee)) \\
\mc E &\mapsto \pi_{2*}(\pi_1^*(\mc E) \otimes \mc P)\nonumber
\end{align}
gives an equivalence of triangulated categories \cite{Mukai_1981}. Furthermore identifying $A^{\vee \vee} \simeq A$ we obtain $\Phi^{\vee}: D^b(\Coh(A^\vee))\to D^b(\Coh(A))$. These satisfy the relation $\Phi^\vee \circ \Phi = \on{inv}^*\circ [-g]$ where $\on{inv}: A\to A$ is the involution on the group structure. For convenience we assume that $A$ is principally polarized with polarization $L$ so that $A \simeq A^\vee$. 
\subsubsection{}
Using that $H^1(A,\Bbbk) = V$ and $H^i(A, \Bbbk) = \Lambda^k(V)$
We can identify 
\[ H^*(A, \Bbbk) =\Bbbk[\Pi V^\vee]  = \Bbbk[\theta_1, \ldots \theta_{2g}] .\] 
Under this identification we  have 
\[ \tau(-) = (-1)^{g/2} \int d\theta (-).\] 
Then we have 
\[ c_1(\mc P) = \sum_i \theta_i \chi_i \] 
where $\{\chi_i\}$ is the basis of $H^1(A^\vee)$ dual to $\{\theta_i\}$ under the identification $H^1(A^\vee) = V^\vee = H^1(A)^\vee$. 
Because $\mc T_A \simeq \mc O_A^{\oplus g}$ we have $\on{Td}(A) = 1$ and so \cite{Huybrechts_2006} there is an equality 
\[ \Phi^{H} = \tau( \pi_1^*(-) \on{ch}(\mc {P})) = \on{PD}\circ (-1)^{\frac{\deg(\deg + 1)}{2} } : H^*(A) \to H^*(A^\vee) =  H^*(A)^\vee = H^*(A)\] 
of the cohomological Fourier-Mukai transform $[\on{ch}(\mc P)]\circ - $ with Poincar\'e duality $\on{PD}$ up to sign. 

Furthermore we have an equality 
\[ \Phi^H(f) = (-1)^{g/2}\widehat{f}(-\sqrt{-1} \chi) \] 
where $\widehat{f}$ denotes the odd Fourier transform \eqref{eq:fourier}. 

Let 
\begin{align*}
- \star - : H^*(A) \otimes H^*(A) \to H^*(A) 
\end{align*}
denote the convolution map where $m: A\times A \to A$ so that $\gamma \star \gamma' = (-1)^{g/2} m_*(\gamma\boxtimes \gamma') = \tau_m(\gamma\boxtimes \gamma')$ is the multiplication map. There is an equality $ \tau_m = \on{PD}\circ m^* \circ \on{PD} $. The operator $m^*$ gives a cocommutative coproduct on $(H^*(A), -\cup -)$ which is part of the Hopf algebra structure on $H^*(A)$ on generators satisfying $m^*(\theta_i) = \theta_i \otimes 1 + 1 \otimes \theta_i$. Since $\langle \lambda, \mu \star \nu \rangle  = \langle m^*\lambda, \mu\otimes \nu\rangle $ we may identify 
$\theta_i \ast -$ with the operator 
$(-1)^{g/2}\partial_{\theta_i}$ on $\Bbbk[\theta_1, \ldots, \theta_{2g}]$.
 Given $f \in \Bbbk[\theta_1, \ldots, \theta_{2g}]$ we have the equality 
\begin{equation}
  \label{eq:FM_is_oddF}
\pd{\Phi^H(f)}{\chi_i}\left(-\chi\right) =\Phi^H(\theta_i f). 
\end{equation}

\subsection{Root systems and singular fibers}\label{ssec:root_systems}

\subsubsection{Weierstrass model} \label{sssec:weier}
For the type $A$ Cartan types in the Deligne series, the relevant equivariant elliptic surface $X_R$ will be studied using the Weierstrass model. The surface $\overline{X}_R$ is a blowup of $\overline{X}'_R$ which cut out by 
\[ Y^2Z = X^3 + a(t)XZ^2 + b(t)Z^3 \] 
from $\PP_{\PP^1_t}(\mc O(2)_X \oplus \mc O(3)_Y \oplus \mc O_Z)$ where $a(t)\in H^0(\mc O(4))$ and $b(t)\in H^0(\mc O(6))$ are taken from the following table: 

\begin{table}[h!]
  \centering
  \begin{tabular}{|l|l|l|}
  \hline
  $R$ &  $a(t)$ & $b(t)$ \\ \hline
     $A_0$       & $0$ & $t$ \\ 
     $A_1$       & $t$ & $0$ \\ 
      $A_2$      &  $0$ & $t^2$ \\ \hline
  \end{tabular}
  \caption{Weierstrass cubic coefficients for equivariant elliptic surfaces.}
  \end{table}

Let $X_R'$ be the restriction of the fibration to $\A^1_t = \PP^1_t \backslash \{ \infty \}$. Then $X_{A_0} = X_{A_0}'$, $X_{A_1} = \on{Bl}_{(0,0,0)}(X_{A_1}')$ is the resolution of an $A_1$ singularity and $X_{A_2} \on{Bl}_{(0,0,0)}(X_{A_2}') \to X_{A_2}'$ is the resolution of an $A_2$ singularity. 

  In the chart $Z = 1$ around the singular point $p_*$ over $t= 0$ with coordinates $(t,x,y)$  the action of $T$ has weights $(a_R, 2, 3)$ where $a_{A_0} = 6, a_{A_1} = 4, a_{A_2} = 3$ so that $e^\hbar [X:Y:Z,t] = [e^{2\hbar}X: e^{3\hbar} Y: Z, e^{a_R \hbar} t ]$.

\subsubsection{Cohomology}

Let $E_a$ denote the fiber of $X_R$ over $a \in \A^1$. 
An identification of $X_R$ with its relative compactified Jacobian fixes a curve $C_0 \subset X_R$ containing the point on $E_0$ corresponding to the trivial line bundle on that fiber. Let $C_1, \ldots, C_r$ denote irreducible curves in $E_0$ so that $C_0, \ldots, C_r$ is a basis for $H_2(X_R, \Z)$ with intersection pairing equivalent to the negative of the Cartan pairing on $\widetilde Q_R$. This $C_0$ is necessarily an affine root with Dynkin label $1$. In the starred cases let $C_\ast$ denote the unique positive dimensional component of $X_R^\csth$ corresponding to the central node of the Dynkin diagram. Let $p_{i} \in X_R^T$ and $p_{ij}\in X_R^T$ respectively denote an isolated fixed point lying only on $C_i$ and lying at the intersection of $C_i$ and $C_j$. In the unstarred cases let $p_*$ denote the unique fixed point which is a singular point of $E_0$, while in the starred cases let $p_*$ be a point on $C_*$.    A basis for $H^*_\csth(X_R)$ is given by $1, D_0, \ldots, D_r$ so that $\langle C_i , D_j\rangle = \delta_{i,j}$. 

We now calculate tangent weights at fixed points in the starred cases. The normal bundle to $C_*$ has weight $\hbar$. 
Given an isolated fixed point $p_{\alpha}\in X_R^\csth$ where $\alpha = i$ or $ij$ let $\ell(p_\alpha) \ge 1$ denote the length of the chain of rational curves connecting $p_\alpha$ to $C_*$. Then the tangent weights at $p_\alpha$ are 
\[ w_\alpha^1 = - \ell(p_\alpha)\hbar  \text{ and }  w_\alpha^2 = (\ell(p_\alpha) + 1)\hbar\]
which follows from the case of an $A_n$ surface. To make formulas nicer we also consider $w_*^1 = 0, w_*^2 = \hbar$ despite the fact that $p_*$ isn't isolated. 

For the unstarred cases we use the Weierstrass model and coordinates on the blownup surface. 
Let $C_0'$ be the singular curve in $X_R'$ and $C_0$ be its strict transform in $X_R$. Calculating the $T$ action in different charts gives the weights at the other $p_i$. We are also able to calculate various cohomology classes in the fixed basis  including $D_i \in H^2_T(X_R, \Z) \simeq  \widehat Q^\vee$ giving a basis corresponding to a set of simple coroots for the affine Lie algebra so that $\langle D_i, C_j\rangle = \delta_{i,j}$ and $\langle D_i, p_* \rangle = 0$ which is identified with the $i$th fundamental coweight $D_i = \omega_i^\vee$ of the affine Lie algebra. The results are summarized in the following tables. 

\begin{table}[h!]
  \centering
  \begin{tabular}{|l|l|l|l | l| }
  \hline
  $R$ &  Fixed pt.    & Coords.  & Weights  \\ \hline
     $A_0$ & $p_*$ & $x,y,t$ &  $w_*^1 = 2\hbar, w_*^2 = 3\hbar$  \\ 
      & $p_0$ & $x,z,t$  &  $w_0^1 = -\hbar, w_0^2 = 6\hbar $ \\ 
     $A_1$    & $p_*$ & $x,y/x,t/x$    & $w_*^1 = \hbar, w_*^2 = 2\hbar$ \\ 
     & $p_0$ & $x,z,t$  & $w_0^1 = -\hbar , w_0^2  = 4\hbar $  \\ 
     & $p_1$ & $x/t,y/t,t$  & $w_1^1 = -\hbar, w_1^2 = 4\hbar$ \\ 
      $A_2$    & $p_*$ & $x,y/x,t/x$    & $w_*^1 = \hbar, w_*^2 = \hbar$   \\
      & $p_0$ & $x,z,t$  & $w_0^1 = -\hbar, w_0^2 = 3\hbar $ \\ 
      & $p_1$ & $x/y,y,t/y-\sqrt{-1}$  & $w_1^1 = -\hbar, w_1^2 = 3\hbar $ \\ 
      & $p_2$ & $x/y,y,t/y + \sqrt{-1}$  & $w_2^1 = -\hbar, w_2^2 = 3\hbar $ \\ 
       \hline
  \end{tabular}
  \caption{Fixed points in resolved Weierstrass models, a system of coordinates vanishing at those fixed points and tangent weights at the fixed points. }
  \end{table}

  \begin{table}[h!]
    \centering
    \begin{tabular}{|l|l|l|l | }
    \hline
    $R$ &  Basis / $H_T(\mrm{pt})$   & Cohomology classes in fixed basis \\ \hline
       $A_0$ & $[p_*], [p_0]$ & $[C_0] = 1/\hbar [p_*]  -1/\hbar [p_0]$   \\ 
       &&$D_0 = \hbar/6 [p_0] $\\
       $A_1$    & $[p_*], [p_0], [p_1]$ & $[C_i] = 1/\hbar[p_*]  -1/\hbar[p_i]$    \\ 
       && $D_i = \hbar/4 [p_i] $\\
       $A_2$    & $[p_*], [p_0], [p_1], [p_2]$ & $[C_i] = 1/\hbar[p_*] - 1/\hbar[p_i]$    \\
       && $D_i = \hbar/3 [p_i] $\\

       $D_4/E_*$ & $[p_*], [C_*], [p_i], [p_{ij}]$ & $[C_*] = [C_*], ~ [C_i] = 1/w_{s(i)}^2[p_{s(i)}] + 1/w_{e(i)}^1[p_{e(i)}] $   \\ 
       && $ D_i = \omega_i^\vee$\\
         \hline
    \end{tabular}
    \caption{Cohomology classes in fixed basis in localized $H^*_T(-)$.  In the $D_4/E_*$ cases the ordering $\beta \succ \alpha$ means that $p_\beta$ lies further from $C_*$ than $p_\alpha$ and $C_i$ is a curve connecting $p_{s(i)}$ to $p_{e(i)}$ where $e(i) \succ s(i)$. } 
    \end{table}

These weights determine the intersection pairing on the surface. Let 
\[ \int_{X_R}: H^*_\csth(X_R) \to \C[\hbar^\pm] \] 
denote the pushforward defined by localization.

\subsubsection{Convolution and Fourier-Mukai transform}\label{ssec:FM_conv}

The Fourier-Mukai transform \eqref{eq:FM} admits a relative version for minimal elliptic surfaces\cite{Bridgeland_1998}. See also \cite{bartocci2009fourier}. By conjugating the cup product with the Fourier transform we can interpret as providing an analogue of a Hopf algebra structure for elliptic surfaces; so that $\star(-,-)$  agrees with the adjoint of the coproduct up to a sign. We only explicitly write a formula for $\Phi^H(-)$ in the equivariant cases over $\A^1$. In types $D_4$ and $E_*$ there are isomorphisms
\[ H^*_{T'\times \Gamma}(E)\simeq H^*_{T'\times \Gamma}(T^*E) \simeq H^*_{T}(X_{R}) \] 
provided by localization and the McKay correspondence, and the relative Fourier transform arises as the action of a element of the derived autoequivalences of the elliptic orbifold $[E/\Gamma]$.

For type $A_0, A_1, A_2$, 
recall that the identification of $X_R$ with its relative compactified Jacobian fixes a simple positive root $\alpha_0$ with Cartan label 1 corresponding to a curve $C_{0}$, and also a identity section $\Theta$ with $\langle \Theta, C_{\alpha_0}\rangle_{\overline S} = 1$ intersecting $C_{\alpha_0}$ in a point $F_0^{sm}\cap S^T$ whose cohomology class is denoted $\mrm{pt}_0$ so $\hbar\Theta = \mrm{pt}_0$ . Thus $\Theta \in \widehat Q^\vee_R$ corresponds to the coroot $d$.   For $q \in Q^\vee \subset \widetilde Q^{\vee}$ we have 
\begin{equation}
  \label{eq:equiv_FM}
  \Phi^H(n \mathrm{pt}_0 + m d + q) = -q + m \mathrm{pt}_0 + (- n ) d .
\end{equation}

\subsubsection{Picard lattice and ample cone} 
\label{ssec:pic_amp_central}
The rational Neron-Severi group $\NS(\overline{X_R})_\Q$ of $\overline{X_R}$ is dual to $\on{NE}(\overline{X_R})_\Q$ which is generated by components of singular fibers at zero and infinity and by sections of the elliptic fibration. In all cases $R$ from \eqref{eq:exceptional_series} we have that $\overline{X_R}$ has a description as a blowdown of a surface $\overline{S_R}$ fitting into a diagram 
\[ % https://tikzcd.yichuanshen.de/#N4Igdg9gJgpgziAXAbVABwnAlgFyxMJZARgBoAGAXVJADcBDAGwFcYkQAdDiWmAJ0ZYwMYAGUA+gCUAviGml0mXPkIpypYtTpNW7ABQBRLngC28AARcAClYB6xAJQB6LgHF6Jk-TkKQGbHgERABMGloMLGyInNy8AkIiABpSsvKKASpE6lQ0EbrRyEYcpvDWdsQuHO6e9JRyWjBQAObwRKAAZnwQJkjqIDgQSMRpIJ3dvTQDSMEjYz2IAMyTg4jDlNJAA
\begin{tikzcd}
{[E\times\PP^1/\Gamma]} \arrow[d] & \overline{S_R} \arrow[ld] \arrow[rd] &                \\
(E\times \PP^1)/\Gamma            &                                      & \overline{X_R}
\end{tikzcd}\] 
where $\Gamma$ is either the automorphism group of $E$ defining $R$ as the Cartan type of the star-shaped Dynkin diagram of the local elliptic orbifold or it is the group  $\Gamma$ corresponding to the dual root system in \eqref{eq:exceptional_series} where duality corresponds to taking centralizer in $E_8$. Concretely we have dual pairs 
\[ (A_0, E_8), (A_1, E_7), (A_2, E_6), (D_4, D_4).\]   
Thus $\overline{X_R}$ is a minimal rational elliptic surface. Also
we have a semiorthogonal decomposition 
\[ D^b(\overline{S_R}) = \langle \pi^* D^b(\overline{X_R}) , \mc A\rangle \]
where $\mc A$ is a Serre subcategory admitting a full strong exceptional collection of length equal to the number of components of the exceptional locus of $\pi: \overline{S_R} \to \overline{X_R}$. We have an equality 
\begin{equation*}
   \NS(\overline{S_R}) = \NS(\overline{X_R}) \oplus K_{num}(\mc A)
\end{equation*}
and an inclusion $\widehat{ R + R'} \hookrightarrow \NS(S_R)$ of the affine root system corresponding to the direct sum of $R$ and $R'$ (the dual root system in the Deligne exceptional series) corresponding to the affine sublattice generated by exceptional curves. The total Picard rank of $\overline{X_R}$ is 10. Owing to the description of $\overline{X_R}$ as the blowup of $\PP^2$ in 9 points we have an identification 
\[\NS(\overline{X_R}) = \Pic(\overline{X_R}) \simeq \RN{1}_{1,9} \]
where $\RN{1}_{1,9} = \langle H, E_1, \ldots, E_9\rangle$ is the intersection pairing with matrix $\on{diag}(1, -1,\ldots, -1)$. In this description the  fiber is class $3H - E_1 - \cdots - E_9$.

\section{Toroidal vertex operator superalgebras}
Let $D_w = w \pd{}{w}$ denote the multiplicative derivative. Given an operator $A$ we let $A^{(j)} = A^j/j!$. 

\subsection{Coordinated modules}

It is a recurring theme in vertex representations of toroidal algebras that the multiplicative fields and OPEs, and relatedly the modified vertex operators, are the most natural formulation. The first indications are the fields of the form $z^kA(z)$ for a field $A(z)$ which show up in the formulas in \cite{Billig_2007}, but see especially \cite{Chen_Li_Tan_2021,Chen_Li_Yu_2022}.

Let $(V, Y, \vac, \omega) $ denote a conformal vertex algebra with central charge $c$. Given $\phi(x,z) = xe^z$ among other choices, \cite{Li_2011} defines the notion of a $\phi$-coordinated module for $V$ to study some of the local fields with multiplicative-style OPEs that arise in the study of vertex algebras such as in \cite{Lepowsky_2000,Zhu_1990}. In the sequel $\phi = xe^z$ will always be assumed. 

\begin{defn}
A $\phi$-coordinated $V$-module is a space $W$ and a linear map 
\[ Y_W(-, z) : V \to \Hom(W, W((z)))\] 
such that $Y_W(\vac, z) = \Id_W$, and for $u,v \in V$ there exists $k\in \Z_{\ge 0}$ such that 
\begin{align*}
(z_1-z_2)^k Y_W(u, z_1) Y_W(v, z_2) &\in \Hom(W,W((z_1,z_2)))\\
(z_2e^{z_0} - z_2)^k Y_W(Y(u,z_0)v, z_2) &= ((z_1-z_2)^k Y_W(u, z_1) Y_W(v, z_2))\big|_{z_1 = z_2e^{z_0}}. 
\end{align*}
\end{defn}

In particular if we have an OPE of the form 
\begin{equation}\label{eq:VOA_OPE}
Y(u, z)Y(v,w) \sim \sum_{j \ge 0} \frac{Y(u_{(j)} v , w)}{(z-w)^{j+1}}
\end{equation}
corresponding to the bracket 
\begin{equation}\label{eq:VOA_bracket}
  [Y(u, z)Y(v,w)] = \sum_{j \ge 0} Y(u_{(j)} v , w)\partial^{(j)}_w \delta(z-w)
  \end{equation}
then given a $\phi$-coordinated module $(W, Y_W)$ one has the commutation relation 
\begin{equation}\label{eq:VOA_coord_bracket}
[Y_W(u,z),Y_W(v,w)] = \sum_{j \ge 0} Y_W(u_{(j)} v, w) D_w^{(j)} \delta\left(\frac{w}{z} \right).
\end{equation}

Following \cite{Zhu_1990}, by letting $\widetilde{\omega} = \omega - \frac{1}{24} c\vac$ and given $v\in V$ defining 
\[ Y[v,z] = Y(e^{zL(0)}v, e^z - 1)\] 
the data $(V, Y[-,z], \vac, \widetilde{\omega})$ defines a vertex algebra and 
\[ T(\vac) = \vac, T(\omega) = \widetilde{\omega}, T(a) = a\] 
for primary $a \in V$ together with compatibility with the Virasoro algebra defines an isomorphism from $(V, Y, \vac, \omega)$ to $(V, Y[-,z], \vac, \widetilde{\omega})$. 

Given $a\in V$ write 
\begin{equation}\label{eq:y_brack}
  Y[a,z] = \sum_{n \in \Z} a[n]z^{-n-1}
\end{equation}
so that the action of $T$ on states formed by action of Fourier coefficients is calculated by
\begin{equation}
\label{eq:t_on_prod}
T(a_n b) = T(a)[n] T(b)
\end{equation}
using that if $a$ is homogeneous of weight $\wt a$ we have 
\begin{equation}\label{eq:a_brackn_formula}
a[n] = \Res_z\left(Y(a,z)\log(1 + z)^n(1+z)^{\wt a - 1}\right).
\end{equation}

Given an (ordinary) $V$-module $(W, Y_W)$ the above definition produces a $\phi$-coordinated module by the following proposition.  
\begin{prop}[{\cite[\S 6]{Chen_Li_Tan_2021},\cite{Lepowsky_2000}}]\label{prop:modified_gives_coord} Define a $\End(W)$-valued field $X(v,z)$ by the formula  
  \begin{equation}\label{eq:X_defn} X^\phi(v, z) := Y_W(z^{L(0)} T(v), z) \end{equation}
The data $(W, X^\phi(-, z))$ defines a $\phi$-coordinated $V$-module structure on $W$. 
\end{prop}

In general, the fields of the form 
\[ X(v, z) = Y(z^{L(0)}v, z) \] 
are known as modified vertex operators (from e.g. \cite{Frenkel_Lepowsky_Meurman_1989}).

As a special case of Proposition \ref{prop:modified_gives_coord} since $V$ is a module over itself, $X^\phi(v, z)$ defines a $\phi$-coordinated $V$-module structure $(V, X^\phi(-,z))$ and we have brackets 
\begin{equation}\label{eq:x_bracket} [X^\phi(u,z), X^\phi(v,w)] = \sum_{j \ge 0} X^\phi(u_{(j)}v, w) D_w^{(j)}\delta\left( \frac w z\right).\end{equation} 

\subsubsection{}
Consider the Lie algebra $\mf g_{A_{-1}}$ from \ref{ssec:gamin1}. Combine its elements into generating series 
\begin{equation}\Upsilon_m(\gamma, z) = \sum_{n \in \Z} w_{\gamma}^{m,n} z^{-n}. \label{eq:wgen}\end{equation}
The defining relation \eqref{eq:main_bracket} 
becomes 
\begin{multline} 
  [\Upsilon_m(\gamma, z), \Upsilon_{m'}(\gamma',w)] = m D_w \Upsilon_{m + m'}(\gamma \star \gamma', w)\delta(w/z) \\
  +(m+m')\Upsilon_{m+m'}(\gamma \star \gamma',w)D_w \delta(w/z) 
  +\delta_{m+m',0} \langle \gamma, \gamma' \rangle \pmb{c}D_w\delta(w/z)
\label{eq:bracket_mult_gen_fn}.
\end{multline}
This may equivalently be expressed via the multiplicative-style operator product expansion 
\begin{multline} 
  \Upsilon_m(\gamma, z) \Upsilon_{m'}(\gamma',w) \sim m D_w \Upsilon_{m + m'}(\gamma \star \gamma', w)\frac{w}{z-w} \\
  +(m+m')\Upsilon_{m+m'}(\gamma \star \gamma',w)D_w \frac{w}{z-w}
  +\delta_{m,-m'} \langle \gamma, \gamma' \rangle \pmb{c}D_w\frac{w}{z-w} 
\label{eq:ope_mult_gen_fn}.
\end{multline}

\subsubsection{Free super bosons}

Recall some of the notation for free bosons valued in a commutative Frobenius algebra from \cite{Maulik_Okounkov_2019} which immediately generalizes to the super setting. Let $\mathbb{H}$ denote a supercommutative Frobenius algebra over a supercommutative ring $\mathbb{K}$. Generally we assume $\mathbb{H} = H^*_G(S/B)_{loc}$ and $\mathbb{K} = H^*_G(B)_{loc}$ where $S$ is a surface over a smooth base $B$. Let $\tau : \mbb H \to \mbb K$ denote the trace, which in the geometric setting is given by $\tau(\gamma) = -\int_{[S]} \gamma$. Likewise the multiplication $m(-,-)$ and supersymmetric form  $\langle -, -\rangle$ are identified with the cup product and the Poincar\'e pairing respectively. 

Consider the Heisenberg-Clifford algebra $\mc H eis_{\mbb H} = \Sym^\bullet(\mbb H[t^{\pm 1}] \oplus \C \pmb c)$ generated by Fourier coefficients of the fields 
\[\widetilde \alpha(\gamma, z) = \sum\widetilde  \alpha_n(\gamma) z^{-n-1}\] 
 for $ \gamma\in \mbb H$. Let $\mc F_{\mbb H, k}$ denote its level $k$ Fock representation, omitting $k$ if $k = 1$. There is a vertex operator algebra structure $(V, Y, |\rangle, \omega)$ on $V = \mc F_{\mbb{H}, k}$ so that for $\gamma \in \mbb{H} \subset \mc F_{\mbb{H}, k}$ of degree $1$ we have $Y(\gamma, z) = \alpha(\gamma, z)$.
The modified vertex operators associated to  $v \in \mbb{H}$ are denoted by
\[ \widetilde{\pmb{\alpha}}(\gamma, z) = \sum \widetilde \alpha_n(\gamma) z^{-n}. \] 

Recall from \cite[\S 13]{Maulik_Okounkov_2019} that given $\Gamma \in \mbb{H}^{\otimes n}$ let $:\pmb{\alpha}^n:(\Gamma, z)$ denote the corresponding normally ordered product and in the particular case that $\Gamma$ is the iterated coproduct $\gamma \mapsto \gamma^{\Delta n} \in \mbb{H}^{\otimes n}$ the field $ :\widetilde{\pmb{\alpha}}^n:(\gamma^{\Delta n}, z)$ is simply denoted $ :\widetilde{\pmb{\alpha}}^n:(\gamma, z).$ The element $\pmb{e} \in \mbb{H}$ so that $m(\gamma^\Delta) = \pmb{e} \gamma$ is called the handle-glueing element. 

\subsubsection{}

Given fields $A(z), B(z)$ we express the terms in the singular part of the OPE using higher products of fields $[A_{j}B]$ as 
\[ A(z)B(w) = \sum_{j \ge 0 }\frac{[A_{j} B](w)}{(z-w)^{j + 1}} \]
 so that the non-commutative Wick's formula \cite[\S 3.3]{Kac_1998} says 
\begin{multline}\label{eq:ncwick}
[A_{m}:BC:](z) = :[A_mB](z) C(z):+ \\
(-1)^{AB}:B(z) [A_mC](z): + \sum_{j = 0}^{m-1} \binom{m}{j} [[A_jB]_{m-1-j}C] (z). 
\end{multline}
Given an OPE calculated using \eqref{eq:ncwick} one immediately deduces the an OPE governing commutation relations between Fourier modes for the modified vertex operators by incorporating appropriate degree shifts. For example the commutation relations between modes $\phi$-coordinated vertex operators $X^\phi(u, z)$ and ordinary fields is calculated by 
\begin{align}\label{eq:xy_ope}
  X^\phi(u, z) Y(v ,w) & \sim\sum_{\substack {s \in  S \\  j \ge 0}} \frac{z^s Y(T_s(u)_{(j)} v  , w)}{(z-w)^{j+1}}
\end{align}
where $S$ is the spectrum of $L(0)$ and $T_s(u)$ is the projection of $T(u)$ onto the $s$-eigenspace of $L(0)$. 

For two $\phi$-coordinated vertex operators we can simply deduce their OPE from \eqref{eq:x_bracket}. 

We also use the special case of Borcherds' identity termed in \cite[\S 4.8]{Kac_1998} quasiassociativity of the normal ordered product, which says 
\begin{equation}\label{eq:quasiassoc}
   ::AB:C:(z) = :A:BC::(z) + 
   \sum_{j \ge 0} [A_{-j-2}[B_jC]] + (-1)^{AB}[B_{-j-2}[A_{j}C]].
\end{equation}
\subsubsection{Normalization convention} 

In the case of the surface $X_{A_{-1}}$ we make a normalization convention 

\begin{align}
\alpha_n([E]) &= \widetilde{\alpha}_n([E])\\
\alpha_n([\pt]) &= \frac{1}{\hbar}\widetilde{\alpha}_n([\pt])\\
\alpha_n([\sigma_{\pm}]) &= \begin{cases}\frac{1}{\hbar}\widetilde{\alpha}_n(\sigma_{\pm}) & n > 0\\
  \widetilde{\alpha}_n(\sigma_{\pm}) & n \le  0\end{cases}
\end{align}
and define fields 
\[ \alpha(\gamma, z), \pmb{\alpha}(\gamma, z)\] 
by linearity which normalizes the coefficients of the bracket on the Heisenberg algebra to Lie in $\Q \subset H^*_T(\pt)$. 

For the other surfaces $X_{R}$ we simply let 
\[ \alpha_k(\gamma) = \widetilde{\alpha}_k(\gamma)\] 
for any $\gamma \in \mathbb{H}$. 

\subsection{Construction of vertex representation}

\subsubsection{}\label{ssec:lattice}

Consider the hyperbolic lattice $\RN{2}_{1,1}$ with intersection pairing $\begin{pmatrix}0&1 \\ 1 & 0 \end{pmatrix}$ in the basis $u,v$. Let $\epsilon: \C[\RN{2}_{1,1}] \otimes \C[\RN{2}_{1,1}] \to \Z$ be the cocycle define by $\epsilon(v,u) = -1$, all other generators $1$ and extended by bimultiplicativity. Given a lattice $L$ and cocycle $\epsilon$ let 
\[ V_\epsilon(L) := \mc F_{L\otimes \C}\otimes \C[L] \]  denote the lattice vertex algebra associated to $L$ and omit the notation $\epsilon$ when unambiguous. 
When $L = \RN{2}_{1,1}^{\oplus n}$ let $L^+ = \langle u_i \rangle_{i=1}^n \subset L$ and define
\[ V^+(\RN{2}_{1,1}^{\oplus n}) := \mc F_{\RN{2}_{1,1}^{\oplus n}\otimes \C} \otimes \C[L^+ ] \] 
which is a sub-VOA of $V(L)$. Consider 
\begin{align}
\label{eq:main_voa} 
V_{T^*E} &= \mc F_{H^*(E)}\otimes \C[L^+]\\
&= V^+(\RN{2}_{1,1}) \otimes \mc F_{H^{\underline 1}(E)}.\nonumber
\end{align}

We will denote $u_1 = E$ using the identification $\RN{2}_{1,1} \simeq \NS(E\times \PP^1)$. 

The vertex operators take the form 
\begin{equation}
  \label{eq:VTE_voa_lattice} 
  Y(\prod_i \alpha_{-n_i}(\gamma_i) e^{mE}, z  ) =  :
  \prod_i \partial_z^{(n_i - 1)} \alpha(\gamma_i, z) Y(e^{mE}, z) :
\end{equation}

\subsubsection{}

Let $\omega(z) = \omega_{H^{\underline 0}(E)}(z) + \omega_{H^{\underline 1}(E)}(z)$ denote the standard conformal field for the Heisenberg part of $V_{T^*E}$ with conformal vector $\nu = \alpha_{-1}(\gamma_i)\alpha_{-1}(\gamma^i)|0\rangle$. For any $m \in \Z$ let $\nu_m = \nu - m \alpha_{-2}([E])|0\rangle$. For any $m$, this gives a conformal vector generating a Virasoro field $\omega_m(z)$ of $c = 0$. Likewise let $\pmb{\omega}_m(z) = X(\nu_m, z) = \sum{L_n}z^{-n}$. 
\[ \Gamma_{mE}(z) = :\exp(\pmb{\phi}(mE, z)): = e^{mE}\sum_{n > 0} : \pmb{\phi}^n:(mE, z)\]
where $\pmb{\phi}(\gamma, z) = \alpha_{\log}(\gamma) + \alpha_0(\gamma)\log z +\sum_{n \neq 0} \frac{\alpha_n(\gamma)}{-n} z^{-n}$ whenever 
$\gamma \in H^{\underline 0}(E, \Z) \oplus H^{\underline 1}(E, \Q)$. 

For $\gamma \in H^{\underline{0}}(E, \Z)$, the constant term $\alpha_{\log}(\gamma)$ is interpreted formally as is standard in the theory of lattice vertex algebras so that $\exp(\alpha_{\log}(\gamma)) = e^\gamma$ in the twisted group algebra of the cohomology lattice used in \ref{ssec:lattice}. For $\gamma \in H^{\underline 1}(E)$ the constant term is defined in \ref{ssec:SF}. 

Then this satisfies 
\[ \pmb{\alpha}(\gamma, z)\pmb{\phi}(\gamma', z) \sim \langle \gamma, \gamma'\rangle \frac{z}{z-w}\]
and furthermore we have $\Gamma_{mE}(z) = Y(e^{mE}, z)$. 

\subsubsection{Symplectic fermions}
\label{ssec:SF}
The vertex algebra $\mc F_{H^{\underline 1}(E)}$ is known as the symplectic fermion model \cite{Kausch_1995,Kausch_2000}.

It is closely related to the $bc$-system. Consider an associative superalgebra generated by odd elements
\begin{equation}\label{eq:asso_bc} 
\mc A_{bc} = \langle b_n, c_n, \theta^b, \theta^c \rangle_{n \in \Z}
\end{equation}
where the only non-zero anticommutators are 
\begin{align*}
[b_n, c_m] &= -n \delta_{n + m, 0} & [\theta^c, b_0] &=1 & [\theta^b, c_0] &= 1. 
\end{align*}

Let $V_{bc}$ denote the vector space freely generated from $\vac_{bc}$ by operators $b_n, c_n$ for $n<0$ and operators $\theta^b, \theta^c$  and annihilated by the other generators. The generating fields 
\begin{align}\label{eq:bc}
  b(z) &= \sum_{n \in \Z} b_n z^{-n-1}  & c(z) & =\sum_{n\neq  0} \frac{c_n}{n} z^{-n} + c_0 \log z + \theta^c
\end{align}
act on $V_{bc}$ and satisfy the OPE
\begin{equation}\label{eq:bc_ope}
  b(z)c(w) \sim \frac{1}{z-w}. 
\end{equation}

The subspace of $V_{bc}$ generated by the modes of $b(z)$ and $c(z)$ is isomorphic to $\mc F_{H^{\underline 1}(E)}$. Furthermore there is an isomorphism 
\[ \Heis_{H^{\underline 1}(E)} \simeq \on{Modes}(\{ b, \partial c\} )\]
where $\on{Modes}(\{\phi_1, \ldots, \phi_k\}) \subset \End(V)$ denotes the associative algebra generated by the Fourier coefficients of a set of $\End(V)$-valued fields by setting 
\begin{align}\label{eq:bc_and_alpha}
  b(z) &= \alpha(\sigma_+, z) &  \partial c(z) &= \alpha(\sigma_-, z)
\end{align} 
so that defining $\alpha_{\log}(\gamma)$ to be the appropriate linear combination of $\theta^b$ and $\theta^c$. Fixing this identification we consider $\phi(\gamma,z)$ to be $\End(V_{bc})$-valued fields. 

We let 
\begin{equation}\label{eq:current_bc}
  J_{bc}(z) = : c(z) b(z) : = \sum_{n \in \Z} J_n z^{-n-1}
\end{equation}
whose modes $J_n$ generate an (even) Heisenberg algebra with $c = 1$ which is the usual boson-fermion correspondence under the identification of the $bc$-ghost system with the free fermion system. 

The field 
\begin{equation}
\omega_{bc}(z) = \omega_{H^{\underline 1}(E)}(z) = :\partial c(z) b(z):
\end{equation}
 differs from the usual conformal field under the identification with free fermions in particular satisfying $c = -2$. Wick's theorem shows 
 \begin{equation} \omega_{bc}(z)J_{bc}(w) \sim \frac{-1}{(z-w)^3} + \frac{J_{bc}(w)}{(z-w)^2} + \frac{\partial J_{bc}(w)}{z-w}
 \end{equation} 
 corresponding to 
 \begin{equation}
\label{eq:j_stress}
[L_n, J_m] = -n J_{m+n} -\delta_{n+m, 0} (n^2 + n) \pmb{c}_{LJ}
 \end{equation}
 where $\pmb{c}_{LJ} = \frac{1}{2}$ acts by $1/2$ times the identity. 

Notice that the field 
\begin{align}
   \overline{\omega}_{bc}(z) &= \omega_{bc}(z) - \partial J_{bc}(z) \nonumber\\
   &= :\partial c(z) b(z): - \partial : c(z)b(z): \label{eq:new_omega}\\
   &=   -: c(z) \partial b(z):\nonumber
\end{align}
is also a stress tensor at $c = -2$ which also arises by exchanging the roles of $b$ and $c$. 

In the higher rank case we consider a rank $g$ bc-system corresponding to a symplectic vector space $T^*A = A\oplus A^\vee$ where $A$ has dimension $g$. Picking a basis $a_1, \ldots, a_g$ for $A$, the rank $g$ $bc$-system $V_{bc}^A$ is defined to be the tensor product 
\[ V_{bc}^A = V_{bc}^{(1)} \otimes\cdots \otimes V_{bc}^{(g)}\] 
with fields $b^{(i)}(z)$ and $c^{(j)}(z)$ satisfying 
\[ b^{(i)}(z) c^{(j)}(w) \sim \frac{\delta_{ij}}{z-w}. \] 
The zero and logarithmic modes are denoted $b_0^{(i)}, c_0{(j)}$ and $\theta^b_{(i)}, \theta^c_{(j)}$. 

\subsubsection{Zero modes of bc-system geometrically}

Let $\Jac(E)_\otimes \simeq E$ be a copy of $\Jac(E)$. There is an action map 
\begin{align*}
   \Jac(E)_\otimes \times M(v; \mc F\otimes \mc P)/\Jac(E) &\xrightarrow{\otimes} M(v; \mc F\otimes \mc P)/\Jac(E)'\\
   (\mc L, (\mc E, \phi) ) &\mapsto (\pi_E^*\mc L \otimes \mc E,(\mc L \otimes -)\circ  \phi \circ (\pi_E^* \mc L^{-1}\otimes - )). 
\end{align*}
Let $\tau_+, \tau_- \in H^1(\Jac(E)_\otimes)$ denote a basis of A and B cycles. Convolution with the graph of the action map agrees with $\otimes_*$ and we let 
\begin{equation}
\gamma \star_\otimes  - : H^*_T(M(v; \mc F\otimes \mc P)/\Jac(E)) \to H^*_T(M(v; \mc F\otimes \mc P)/\Jac(E))
\end{equation}
denote the map $x \mapsto \otimes_*(\gamma \boxtimes x)$. 

Let $\xi_+, \xi_- \in H^1(\Jac(E))$ denote the classes of A and B cycles in the cohomology of the base of $M(v; \mc F)/\Jac(E)$. 

Let $\Heis^0_{H^{\underline{1}}(E)}\subset \Heis_{H^{\underline{1}}(E)}$ denote the subalgebra spanned by the zero modes $c, \alpha_0(\sigma_+), \alpha_0(\sigma_-)$. This is contained in $\Heis^{0, \log}_{H^{\underline{1}}(E)}$ which also contains $\alpha_{\log}(\sigma_+)$ and $\alpha_{\log}(\sigma_-)$. 
We assign 
\begin{align}\label{eq:zero_fermi}
\alpha_0(\sigma_\pm) &\mapsto \xi_{\pm} \cup - \\
\label{eq:log_fermi}
\alpha_{\log}(\sigma_\pm) &\mapsto \tau_\pm \star_\otimes - . 
\end{align}

Now define 

\begin{align}\label{eq:VJAC}
   V_{T^*E}^{\Jac} &:= \bigoplus_{n,m} H^*_T(M(v + n[\pt] + m[E]; \mc F\otimes \mc P)/\Jac(E))\\
   &= V_{T^*E}\otimes H^*(\Jac(E))\nonumber \\
   &= V^+(\RN{2}_{1,1})\otimes V_{bc}.\label{eq:vte_jac_is_tensor}
\end{align}

\begin{prop}\label{prop:bc_zero_modes_geo}
The assignments \eqref{eq:zero_fermi}-\eqref{eq:log_fermi} together with the Nakajima action give a representation of the modes of $\pmb\phi(\sigma_\pm, z)$ 
on $V_{T^*E}^{\Jac}$. 
\end{prop}
\begin{proof}
These operators commute with the Nakajima operators and so we only need to show their relations with each other which is straightforwardly verified using that \[ - \star - = \Phi(\Phi^{-1}(-) \cup \Phi^{-1}(-)).\] 
\end{proof} 

\subsubsection{Higher genus} 

For more general ruled surfaces $S_C = \PP(\mc T_C\oplus \mc O_C)$ pullback from the base curve $C$ gives an identification 
\[ \Pic_0(S_C)\simeq \Jac(C). \] 
Given $\mc L \in \Jac(C)$ 
there is a relative version of the Hilbert scheme of points in $T^*C$ corresponding to rank 1 torsion-free sheaves with a framing at $D_\infty = S_C\backslash D_\infty$. The determinant map $M(v) \to \Pic(S_C)$ restricts to give a relative moduli space
\[ M(1,0,-n; \mc L)/\Jac(C) .\]

Let $\sigma_{+, i}, \sigma_{-,i}$ for $i = 1, \ldots, g$ denote a basis of 
$H^1(C)$ of $A$ and $B$ cycles so that 
$\langle \sigma_{+,i}, \sigma_{-},j\rangle = \delta_{i,j}$. Let $H^*(\Jac(C)) = \C[\xi_{\pm, i}]$ be the corresponding cohomology ring of the Jacobian and $\tau_{\pm, i}$ generators for $\Jac(C)_\otimes$ which acts by tensor product. 

Consider the assignment 
\begin{align*}\alpha_0(\sigma_{\pm,i}) &\mapsto \xi_{\pm,i} \cup - \\
  \alpha_{\log}(\sigma_{\pm, i}) &\mapsto \tau_{\pm, i}\star_\times -.\end{align*}

  The above argument for the elliptic curve case in genus 1 immediately generalized to higher genus. 
\begin{prop}
The above assignment together with the Nakajima action gives an identification 
\[ \mc F_{H^{\underline 0}(C)} \otimes V_{bc}^{H^{1,0}(C)}\simeq \bigoplus_n H_T^*(M(1,0,-n; \mc L)/\Jac(C) ) \] 
of the cohomology of moduli spaces of rank 1 sheaves with the tensor product of a rank 2 free boson and a rank $g$ $bc$-system. 
\end{prop}

\subsubsection{Variant Lie superalgebra}

There is a variant of $\mf g_{T^*E}$ denoted $\mf g_{T^*E}^{\circ}$ whose generating fields satisfy OPEs of the form \eqref{eq:VOA_OPE}. This is the analogue of the algebra $\widehat{\mf t}(\mf g, \mu)^\circ$ from \cite{Chen_Li_Tan_2021}. 

Writing $\deg(\gamma)$ for the cohomological degree so that $\deg(E) = 0, \deg(\sigma_\pm) = 1, \deg(\pt) = 2$ for $\gamma \in B := \{ E,\sigma_-,\sigma_+, \pt \}$, let  
\begin{equation} 
\mf g_{T^*E}^\circ = \langle\overline{w}_\gamma^{a,b}, \pmb c\rangle_{\gamma \in B}^{a,b \in \Z}
\end{equation}
with bracket 
\begin{equation}
  \label{eq:wcirc_bracket} 
  [\overline{w}_\gamma^{a,b}, \overline{w}_\eta^{c,d}] = \kappa(a,b,c,d;\gamma, \eta) w_{\gamma\star \eta}^{a + c, b + d} + b\langle \gamma, \eta\rangle \delta_{a + c, 0} \delta_{b + d, 0} \pmb c
\end{equation} 
where 
\[ \kappa(a,b,c,d; \gamma, \eta) = -\det\begin{pmatrix}a & b \\ c & d \end{pmatrix} - a\deg(\gamma \star \eta) + (a + c)(\deg(\gamma)  -1)\] 
which we encode in the generating series 

\[ \overline{\Upsilon}_m(\gamma, z) = \sum_{n \in \Z}  \overline{w}_\gamma^{m,n}z^{-n-\deg{\gamma}} \] 
for homogeneous $\gamma$. 

Then \eqref{eq:wcirc_bracket} is equivalent to the OPE 
\begin{multline}\label{eq:upsilon_circ_ope} \overline{\Upsilon}_m(\gamma, z) \overline{\Upsilon}_{m'}(\gamma, w)  \sim \\ m\frac{\partial_w \overline \Upsilon_{m + m'}(\gamma \star \gamma', w)}{z-w} 
  +(m+m')\frac{\overline \Upsilon_{m+m'}(\gamma \star \gamma',w)}{(z-w)^2}
  +\frac{\delta_{m,-m'} \langle \gamma, \gamma' \rangle \pmb{c}}{(z-w)^2}.  
\end{multline}

\subsubsection{Lattice realization of full Jacobian algebra} 

We will realize a lattice realization of the Lie algebra from Section \ref{ssec:full_jacobian_lie} using the construction of Billig \cite{Billig_2006} of the full toroidal Lie algebra. 

Let $\mf g$ be a semisimple simply laced Lie algebra and $V_{\mf g, 1}$ the level $1$ representation which admits a description as a lattice VOA. We again allow $\mf g = \mf {gl}_0$ in which case $V_{\mf g, 1}$ only contains the identity field.  Denote its conformal vector $\omega_{\mf g}$ with central charge $k_{\mf g}$ which acts in $V_{\mf g, 1}$ via 
\[\frac{\ell\dim(\mf g)}{\ell + h^\vee} = \frac{\dim \mf g}{1 + h^\vee}. \] 
Consider the VOA $V_{HVir}$ freely generated by a Virasoro field $\omega_{aux}(z)$ with central element $k_V$ and a Heisenberg field $I(z) = Y(I, z)$ with central element $k_H$ with twisted central element $k_{HV}$ corresponding to the OPE 
\[ \omega_{aux}(z) I(w) \sim \frac{-2k_{HV}}{(z-w)^3} + \frac{I(w)}{(z-w)^2} + \frac{\partial I(w)}{z-w}.\] 

Recall the VOA $V^+(\RN{2}_{1,1})$ with conformal vector $\omega_{\RN{2}_{1,1}}$. We make the identification \eqref{eq:main_voa} induced by $\NS(E\times \PP^1) = \RN{2}_{1,1}$.  In the special case we are considering, Billig's Theorem produces a vertex representation of the toroidal algebra as follows.  Letting 
\[ \omega = \omega_{\mf g} + \omega_{aux} + \omega_{\RN{2}_{1,1}}\] consider the assignments 
\begin{align}
\mrm{k}_a(z) &= \begin{cases}
  \frac 1 a Y(e^{aE}, z) & \text{   if  } a \neq 0 \\
  \pmb\phi ([E], z) - E - \alpha_0(E)\log z & \text{  if } a = 0 \end{cases}\label{eq:billig_K} \\
  J_{x,a}(z) &= :Y(x, z) Y(e^{mE}, z): \\
  \mrm{d}_s(z) &= :\alpha([\pt], z)Y(e^{mE}, z):+ m:I(z)Y(e^{mE}, z):\\
  \mrm{d}_t(z) &= :Y(\omega,z)Y(e^{mE}, z):+ :I(z)\alpha(m[E],z)Y(e^{mE}, z): \notag\\
    &- :\partial_z Y(m[E], z)Y(e^{mE}, z):\label{eq:billig_Dt}
\end{align}

\begin{thm}[{\cite{Billig_2006}}]\label{thm:billig_full}
Let $V$ be a module for $V_{\mf g,1} \otimes V^+(\RN{2}_{1,1})\otimes V_{HVir}$ where the central charges act via 
\begin{align*}
k_V &= -2 & k_{H} &= 1 & k_{HV} &= \frac{1}{2}
\end{align*}
then \eqref{eq:billig_K}-\eqref{eq:billig_Dt} give a representation of $\mc T(\mf g)$  on $V$. 
\end{thm} 

There is a specific representation 
\begin{align}
  V_{HVir} &\to V_{bc}((z)) \label{eq:hvir_rep}\\
  \omega_{aux} &\mapsto \overline{\omega_{bc}}(z) \notag \\
  I(z) & \mapsto - J_{bc}(z) \notag 
\end{align}
satisfying the hypotheses of Theorem \ref{thm:billig_full}. We conclude 

\begin{prop}
The VOA $V_{T^*E}^{\Jac}$ is a module for the full toroidal algebra. 
\end{prop}
\begin{proof}
Immediately follows from the description of $V_{T^*E}^{\Jac}$ in \eqref{eq:vte_jac_is_tensor} and Theorem \ref{thm:billig_full} under the representation \eqref{eq:hvir_rep}. 
\end{proof}

Consider two additional fields 
\begin{align}\label{eq:extra_sigmapm}
\sigma^+_a(z) &= : \alpha(\sigma_+, z)Y(e^{aE},z): & \sigma^-_a(z) &= : \alpha(\sigma_-, z)Y(e^{aE},z):
\end{align}

\begin{defn}\label{defn:Ubullet}
The algebra $U^\bullet(\mf g_{T^*E}^{\Jac{}})$ is the subalgebra of the algebra of modes of $V_{T^*E}^{\Jac}$ generated by the modes of the fields \eqref{eq:extra_sigmapm} and \eqref{eq:billig_K}-\eqref{eq:billig_Dt}. 
\end{defn} 

\subsubsection{Lattice realization of variant Lie algebra}

We now give fields in $V_{T^*E}$ whose OPEs satisfy \eqref{eq:upsilon_circ_ope} giving a representation of $\mf g_{T^*E}^\circ$ on $V_{T^*E}$. 

Define fields 
\begin{align}
  \label{eq:d_circ_freefld}
  \overline{\mrm{d}}_m(z)
  &=   m : \omega(z) \Gamma_{mE}(z) : -\partial_z : \alpha(\pt, z)\Gamma_{mE}(z): -m^2 :\partial_z \alpha(E, z) \Gamma_{mE}(z):\\
  \label{eq:k_circ_freefld}
  \overline{\mrm{k}}_m(z) &= \begin{cases}\frac{1}{m}\Gamma_{mE}(z) & m \neq 0 \\
    \pmb{\phi}(E, z)-E-\alpha_0(E)\log z & m = 0\end{cases}\\
    \overline\sigma^+_m(z) &= :\alpha(\sigma_+, z)\Gamma_{mE}(z):\\
    \label{eq:sm_circ_freefld}
    \overline \sigma^-_m(z) &= :\alpha(\sigma_-, z)\Gamma_{mE}(z):.
\end{align}

\begin{thm}\label{thm:variant_LA}
The assignment \eqref{eq:d_circ_freefld}- \eqref{eq:sm_circ_freefld}
gives a representation of $\mf g_{T^*E}^\circ$ on $V_{T^*E}$. 
\end{thm}
\begin{proof}
  Because of the formula 
  \begin{equation}
  \label{eq:getting_back_OG_omega} 
  \omega_{bc}(z) =\overline{\omega_{bc}}(z) - (-\partial J_{bc})(z)
  \end{equation}
  the fields $\overline{\mrm d}_m(z)$ and $\overline {\mrm{k}}_m(z)$ satisfy the relations in the variant of the toroidal extended affine Lie algebra of $\mf{gl}_0$, denoted  $\mf t(\mf{gl}_0)^\circ$ on account of \cite[\S 5]{Chen_Li_Tan_2021} which uses the stress tensor \eqref{eq:getting_back_OG_omega} in the fermionic factor of $V_{T^*E}$. The relevant calculations are also close to those in \cite{Billig_2007}. This provides the commutation relations between the bosonic fields.

For the fermionic fields, using the vanishing of the central charge of the field $\omega(z)$ and the noncommutative Wick formula we obtain the following 
\begin{align*}
  \alpha(\gamma, z)\alpha(\gamma', w)&\sim \frac{\langle\gamma, \gamma'\rangle}{(z-w)^2 }\\
   \alpha(\gamma, z)\Gamma_{mE}(z)&\sim \frac{\langle\gamma, mE \rangle\Gamma_{mE}(w)}{z-w } \\
   :\alpha(\sigma_+, z)\Gamma_{mE}(z)::\alpha(\sigma_-, w)\Gamma_{m
   'E}(w):&\sim \frac{m}{m+m'} \frac{\partial \Gamma_{(m+m')E}(w)}{z-w}  + \frac{\Gamma_{{(m+m')}E}(w)}{(z-w)^2} \\
   :\omega(z)\Gamma_{mE}(z)::\alpha(\sigma_\pm,w)\Gamma_{m'E}(w):&\sim  \frac{\partial :\alpha(\sigma_\pm, w)\Gamma_{(m+m')E}(w):}{z-w}  \\
   & ~~~ + \frac{:\alpha(\sigma_\pm, w)\Gamma_{(m+m')E}(w):}{(z-w)^2} \\
   : \alpha(\pt, z) \Gamma_{mE}(z):: \alpha(\sigma_{\pm}, w)\Gamma_{m'E}(w): &\sim m' \frac{:\alpha(\sigma_\pm, w)\Gamma_{(m+m')E}(w):}{z-w} 
\end{align*}
for which we also use the following formulae which follow from the vanishing of the central extension of the Heisenberg algebra of an isotropic vector space:
\begin{align*}
  : \Gamma_{mE}(z)\Gamma_{m'E}(z): &= \Gamma_{(m+m')E}(z)\\
:\partial \Gamma_{mE}(z) \Gamma_{m'E}(z): &= \begin{cases} \frac{m}{m+m'}\partial \Gamma_{(m+m')E}(z) & m \neq -m' \\
  m  \alpha(E, z) & m = -m'\end{cases}
\end{align*}
which establishes the OPE \eqref{eq:upsilon_circ_ope}. 
\end{proof}

\begin{rmk}
The OPEs between the fields $\overline{\mrm d_m}$ and $\overline{\mrm d_{m'}}$ may be computed directly using Wick's formula \eqref{eq:ncwick} and quasiassociativity \eqref{eq:quasiassoc}, the latter of which implies 
\begin{align*}
::\omega \Gamma_{mE} :\Gamma_{m'E}:(z)  &= :\omega \Gamma_{(m+m')E}:(z) + \partial \Gamma_{(m + m')E} (z) + :\partial \Gamma_{mE} \partial \Gamma_{m'E}:(z)\\
::\alpha(\pt, z)\Gamma_{mE}(z):\Gamma_{m'E}:(z) &= :\alpha(\pt, z) \Gamma_{(m+m')E}(z): +m'\Gamma_{(m+m')E}(z)\\
::\partial \alpha(\pt, z)\Gamma_{mE}(z):\Gamma_{m'E}:(z) &= :\partial \alpha(\pt, z) \Gamma_{(m+m')E}(z): -m'\Gamma_{(m+m')E}(z).
\end{align*} 
\end{rmk}
\subsubsection{Algebra from coordinated vertex operators}

Now we use Proposition \eqref{prop:modified_gives_coord} to produce a module for $\mf g_{T^*E}$. The operators $X^\phi(v,z)$ corresponding to the vertex operators \eqref{eq:d_circ_freefld}- \eqref{eq:sm_circ_freefld} are calculated from the definition to be
\begin{align}
  \label{eq:d_freefld}
  \mrm{d}_m(z)
  &=  mz^2:\omega(z) \Gamma_{mE}(z):  + m D_z \Gamma_{mE}(z) - D_z \left[ z: \alpha(\pt,z)  \Gamma_{mE}(z):\right] \\
  \nonumber
  &- m  z^2: \partial_z \alpha(mE, z) \Gamma_{mE}(z):\\
  \label{eq:k_freefld}
  \mrm{k}_m(z) &= \begin{cases}\frac{1}{m}\Gamma_{mE}(z) & m \neq 0 \\
    \pmb{\phi}(E, z)-E-\alpha_0(E)\log z & m = 0\end{cases}\\
    \sigma^+_m(z) &= :\pmb{\alpha}(\sigma_+, z)\Gamma_{mE}(z):\\
    \label{eq:sm_freefld}
    \sigma^-_m(z) &= :\pmb{\alpha}(\sigma_-, z)\Gamma_{mE}(z):.
\end{align}

Then from \eqref{eq:x_bracket} it follows that the fields \eqref{eq:d_freefld} - \eqref{eq:sm_freefld} satisfy the OPEs 
\begin{align}\mrm {d}_m(z) \mrm{d}_{m'}(w) &\sim  m\frac{D_w \mrm d_{m + m'}(w)w}{z-w} +(m+m')\frac{ \mrm d_{m + m'}(w)zw}{(z-w)^2} \label{eq:ddope}\\
  \mrm{d}_m(z) \mrm{k}_{m'}(w) &\sim m \frac{D_w \mrm{k}_{m+m'}(w)w}{z-w} + (m + m')\frac{\mrm{k}_{m+m'}(w)zw}{(z-w)^2} + \delta_{m + m',0}\frac{\mathbb{c} zw}{(z-w)^2}\label{eq:dkope}\\
\mrm{d}_m(z) \sigma^\pm_{m'}(w) &\sim m \frac{D_w \sigma^\pm_{m+m'}(w)w}{z-w} +(m + m')\frac{\sigma^\pm_{m+m'}(w)zw}{(z-w)^2}\label{eq:dsope} \\
\sigma^+_m(z) \sigma^-_{m'}(w) &\sim  m \frac{D_w \mrm{k}_{m+m'}(w)w}{z-w} + (m + m')\frac{\mrm{k}_{m+m'}(w)zw}{(z-w)^2} + \delta_{m + m',0}\frac{\mathbb{c} zw}{(z-w)^2}.\label{eq:ssope}
\end{align} 
We will combine the above fields into a $H^*(E)$-valued field $\mathcal{D}_m(z) = (\mrm{k}_m, \sigma^+_m, \sigma^-_m, \mrm{d}_m)(z)$ in the basis $\{E, \sigma_a, \sigma_b, \mrm{pt}\}$ and write $\mc{D}(\gamma, z)$ for the $\gamma \in H^*(E)$ component of $\mc D(z)$. We can then rewrite \eqref{eq:ddope}-\eqref{eq:ssope} as 
\begin{multline}
\mc{D}_m(\gamma, z) \mc{D}_{m'}(\gamma',w) \sim m D_w \mc{D}_{m + m'}(\gamma \star \gamma', w)\frac{w}{z-w} \\
+(m+m')\mc{D}_{m+m'}(\gamma \star \gamma',w)D_w \frac{w}{z-w}\\
+\delta_{m,-m'} \langle \gamma, \gamma' \rangle D_w\frac{w}{z-w} \label{eq:Dlabel_ope}
\end{multline}
exactly matching \eqref{eq:ope_mult_gen_fn}.

\begin{thm}\label{thm:voa_is_gterep}
  The vertex algebra $V_{T^*E}$ is an irreducible representation of $\mf g_{A_{-1}}$. 
  \end{thm}
\begin{proof}
The identification of \eqref{eq:ope_mult_gen_fn} with \eqref{eq:Dlabel_ope} demonstrates that $V_{T^*E}$ is a representation and the irreducibility of the Fock module as a representation of the Heisenberg algebra together with the nonvanishing of the Fourier coefficients of $\mc D(z)$ implies irreducibility. 
\end{proof}

\subsection{Determination by Heisenberg commutator}

Here we show that the field $\mc D_m(\Gamma, z)$ is determined by its commutator with elements of $\mc Heis_{1/0}$. 

Recall the subspace $\End(V_{T^*E})_m \subset \End(V_{T^*E})$ from \ref{ssec:savo} and let $\End(V_{T^*E})_{m,n} \subset \End(V_{T^*E})_m$ denote the subspace of $\Phi$ such that $\Phi(v) \in {V_{T^*E}}_{\alpha + (m,n)}$ if $v \in {V_{T^*E}}_\alpha$. 

\begin{thm}
Let $\mc D'_{m}(\gamma, z)$ for $m \in \Z^\times$ be a field such that $\mc D'_{m}(\gamma, z)[z^{-n}] \in \End(V_{T^*E})_{m,n}$ for all $n \in \Z$. If 
\[ [\mc D_0(\gamma, z), \mc D'_{m}(\eta, w)] = m \mc D'_{m}(\gamma \star \eta, w) D_w \delta(w/z) \] 
then $\mc D'_m(\eta, z) = c \mc D_m(\eta, z)$ for some $c \in \Q$. 
\end{thm}

\begin{proof}
For the $\mc D'_m(E,z)$ component of $\mc D'_m$ this is a minor variation on a standard fact in the theory of VOAs, e.g. \cite[\S 5.2]{Kac_1998}.  Likewise the $\sigma_a$ and $\sigma_b$ components $\mc D'_{m}(\sigma_{a/b}, z)$ are determined by the argument in \cite[\S 2]{Heluani_Kac_2007} or by a simpler version of the argument below. Furthermore the constant $c$ must be the same for each component on account of the commutator with the $\alpha_k(\sigma_{a/b})$ Fourier coefficients of $\mc D_0(\sigma_{a/b}, z)$. 

We will now show that the commutation relations with Heisenberg generators determine $\mc D_m(A, z)$ up to the same multiplicative constant $c$. Given $\gamma \in \{ E, \sigma_a, \sigma_b, \mrm{pt}\}$  and $\ell \ge 0$ let $U_{\gamma, \ell} \subset \mc U(\mc Heis_{H^*(E)})$ denote the span of element of the form $:\alpha_{k_{1}}(\gamma)\cdots \alpha_{k_{\ell}}(\gamma) x : $ for $x \in U(\mc Heis_{\gamma^\perp})$. Let $\mf{glf}_{\gamma, \ell}(V_{T^*E})\subset \mf{glf}(V_{T^*E})$ denote the space of fields all of whose coefficients are approximated by sequences of elements of $\C[L^+]\otimes U_{\gamma, \ell}$ in the usual topology. Then 
\[ [\alpha_k(E), \mc D'_m(\mrm{pt}, z)] = z^k km\mc D'_m(E, z) \] 
which implies that $\mc D'_m(\mrm{pt}, z) \in \mf{glf}_{\mrm{pt}, 1}(V_{T^*E})$. Using the same observation we know that
\[ \mc D'_m(\mrm{pt}, z) = \sum_{j \in \Z} :\alpha_{j}(\mrm{pt})q_j(z): + \sum_{k,\ell \in \Z} : \alpha_k(\sigma_a)\alpha_j(\sigma_b) r_{k, \ell}(z): + s(z) \] 
where $q_j(z), r_{k, \ell}(z), s(z) \in \mf{glf}_{\gamma, 0}(V_{T^*E})$ for all $\gamma \in E^{\perp}$ (i.e. all coefficients are "built from $\alpha_k(E)$ only") and further that $q_j(z)$ and $r_{k, \ell}(z)$ are uniquely determined by the brackets $[\alpha_k(\gamma), \Omega_m(\mrm{pt},z)]$ with $\gamma = E$ and $\gamma =  \sigma_a$ respectively. Finally notice that if $s'(z) = :s(z)\Gamma_{-mE}(z):$ then 
\[ [\alpha_{k}(\mrm{pt}) , s'(z)] = cm^2k(1-k)z^k \] 
which determines $s(z) = : s'(z)\Gamma_{mE}(z):$.  
\end{proof}

\subsection{Other surfaces}
For the toroidal extended affine Lie algebras outside of type $A_{-1}$, the geometry of the Hilbert schemes no longer provides us with the symplectic fermion VOA and therefore we lose access to the Virasoro algebra at $c = -2$. By geometric arguments and general expectations concerning cohomological Hall algebras of surfaces, where it is expected that the cohomological Hall algebra of a surface contains the universal enveloping algebra of a Lie algebra, we expect a construction of an algebra with the same character as the toroidal extended affine algebra of type $R$ on 
\[V_{X_R} = \bigoplus_{\alpha \in R^{ell}} H^*_T(M(v + [\alpha]; \mc F)).\]

As a vector space, the first chern character of vectors of the type $v  + [\alpha]$ gives a sublattice of $\NS(\overline{X_R})$ of type $\widehat{R}$ giving an identification as vector space 
\[ V_{X_R} \simeq \mc F_{H^*(X_R)} \otimes \C_\epsilon[\widehat{R}]\] which is futhermore a sublattice of a tensor product of a free bosonic system with a sub-VOA of $V(\RN{1}_{1,9})$, the lattice VOA associated to the lattice $\RN{1}_{1,9}$. 

We embed this into the VOA 

\[ V_{\overline{X_R}} = \mc F_{H^*(\overline{X_R})} \otimes \C_\epsilon [\widehat{R}] \] 
defined as a sub-VOA of the Neron-Severi lattice VOA associated to the surface $\overline{X_R}$. 

Choose a curve 
\[ C \subset D_\infty \subset \overline{X_R} \]
such that $\langle C, C\rangle \neq 0$ and $\langle C, E\rangle = \langle C, \Theta\rangle = 0$ where $\Theta$ is the identity section of the identification of $\overline{X_R}$ with its relative compactified Jacobian and $E$ is the fiber class. Let  $\mc C  =  k[C] \in H^*(\overline{X_R})$ denote the normalized class so that $\langle \mc C, \mc C\rangle = 1$. 

Pick a class 
\[ d = [\Theta] + k_d [E] \] 
such that $\langle d, d\rangle = 0, \langle d, R\rangle = 0$ and $\langle d , [E]\rangle = 1$. 

Let $Q_\Theta^\ell \subset H^*_T(\overline{X}_R)_{loc}$ denote the span of the classes supported on $F_0$ and by $d$. This induces a decomposition 
\begin{equation}\label{eq:decomp1} Q_\Theta^{\ell}\oplus Q_\Theta^{\ell, \perp} = H^*_T(\overline{X}_R)_{loc}\end{equation}
such that $\mc C \in Q_\Theta^{\ell, \perp}$.

The conformal vector $\omega$ for $V_{\overline{X}_R}$ admits a decomposition 
\[\omega = \omega_{R} + \omega_{[E], d} + \omega_{\perp} \] 
based on the decomposition \eqref{eq:decomp1} and the further decomposition 
\[ Q_\Theta^\ell = R\otimes \C \oplus \C\langle[E], d \rangle\]
establishing a tensor decomposition
\[ V_{\overline{X_R}} = V_{\mf g_R, 1} \otimes V^+(\RN{2}_{1,1}) \otimes \mc F_{Q^{\ell, \perp}_{\Theta}}. \]  
Further we can define the conformal field 
\[ \omega_{\perp, \mc C}(z) = \omega_{\perp}(z) - \partial \alpha([\mc C],z). \] 
This gives a representation 
\begin{align}
   V_{HVir} &\to \mc F_{Q^{\ell, \perp}_{\Theta}}\label{eq:norm_1_heis}\\
   \omega_{aux}(z) &\mapsto \omega_{\perp, \mc C}(z) \notag \\
   I(z) &\mapsto \alpha([\mc C], z). \notag 
\end{align}
\begin{thm}\label{thm:goofy_other_surfaces}
The assignments \eqref{eq:billig_K}-\eqref{eq:billig_Dt} give a representation 
\[ \mc T (\mf g_R) \to \End(V_{\overline{X_R}})\]
inducing a representation 
\[\mf t(\mf g_{R}) \to  \End(V_{\overline{X_R}}).\]
\end{thm} 
\begin{proof}
This follows from Theorem \ref{thm:billig_full} and the induced construction of the toroidal extended affine Lie algebra using $\phi$-coordinated modules which in this instance is a direct consequence of \cite{Billig_2007} in the language of \cite{Chen_Li_Tan_2021}. 
\end{proof} 

For geometric reasons we hope to get an action of a Lie algebra with the same root spaces and root multiplicities at $\mf t(\mf g_R)$ to act on $V_{X_R}$. It is possible that the representation of Theorem \ref{thm:goofy_other_surfaces} is appropriate for the purposes of enumerative geometry after projecting along the natural map 
\[ V_{\overline{X_R}} \to V_{X_R} \] induced by pullback in cohomology along the inclusion 
\[ M(v; \mc O) \to \overline{X_R}^{[n(v)]}.\] 

\section{Logarithmic and superanalytic structures}
\subsection{Jacobian Lie algebra representation} 
Recall the Jacobian toroidal Lie algebra $\mf g_{T^*E}^{\Jac{}}$ and its subalgebra $\mf g_{T^*E}^{\Jac{}\tilde{}}$ from Section \ref{ssec:jac_g}.
\subsubsection{Representation on $V_{T^*E}^{\Jac{}}$}

\begin{thm}\label{thm:jac_irrep_thm}
The vertex algebra $V_{T^*E}^{\Jac}$ is an irreducible representation of $\mf g_{T^*E}^{\Jac{}}$.
\end{thm}
\begin{proof}
The proof is exactly the same as Theorem \ref{thm:voa_is_gterep} except that we in addition have the actions of the zero and logarithmic modes of the fields 
\[ \pmb{\phi}(\sigma_+, z), ~~~ \pmb{\phi}(\sigma_-, z)\] 
under the identification of Proposition \ref{prop:bc_zero_modes_geo}. 
\end{proof}

\subsubsection{Self extensions}

The structure of the self-extensions of $V_{T^*E}$ and the non-semisimplicity of the module category of $V_{T^*E}$ is related to the identification of $\mc F_{H^{\underline{1}}(E)}$ with the symplectic fermion VOA. 

In particular while Theorem \ref{thm:jac_irrep_thm} implies the irreducibility o $V_{T^*E}^{\Jac}$ under $\mf g_{T^*E}^{\Jac{}}$ this result fails if we consider the representations as a vertex algebra module for $V_{T^*E}$ or as a module for $\mf g_{T^*E}^{\Jac{}\tilde{}}$. 

In particular, the vacuum module $V_{T^*E}$ has a $2$-dimensional space of self-extensions 
\[\C\langle \alpha(\sigma_a)_0, \alpha(\sigma_b)_0\rangle \] induced by those of the vacuum module of $\mc F_{H^{\underline 1}(E)}$. Given $\xi \in \C\langle \alpha(\sigma_a)_0, \alpha(\sigma_b)_0\rangle$ let 
$\mc E_{\xi}$ denote the highest weight $\mc{H}eis_{H^{\underline{1}}(E)}$-module annihilated by $\xi$ and let 
\[V_{\xi} = V^+(\RN{2}_{1,1})\otimes \mc E_{\xi}\]
denote the corresponding $\mf g_{T^*E}^{\Jac{}_0}$ module. The exact sequence 
\[0 \to \mc F_{H^{\underline 1}(E)}\to \mc E_\xi \to \mc F_{H^{\underline 1}(E)}\to 0\] induces an exact sequence 
\begin{equation}
  \label{eq:vte_self_ext}
  0 \to V_{T^*E} \to V_\xi \to V_{T^*E} \to 0 .
\end{equation}

of modules for $\mf g_{T^*E}^{\Jac{}\tilde{}}$ or of VOA modules of $V_{T^*E}$. 

This example highlights the fact that many of the features of the symplectic fermion model, e.g. non-diagonalizability of $L_0$, logarithmic correlation functions, extend naturally to those of the VOA denoted $V_{T^*E}$ 

In particular more generally we have an infinite projective resolution 

\begin{equation}
\cdots \to \C^n\otimes V_{T^*E}^{\Jac} \to \cdots \to \C^2\otimes V_{T^*E}^{\Jac} \to V_{T^*E}^{\Jac} \to V_{T^*E} \to 0 
\end{equation} 
of $V_{T^*E}$ as a VOA module over itself or as a $\mf g_{T^*E}^{\Jac{}\tilde{}}$-module inherited from the one on the symplectic fermions \cite{Creutzig_Ridout_2013}. 

\subsubsection{Superanalytic vertex operators and two variables}\label{ssec:savo}
Let $\{E, \sigma_a, \sigma_b, \rm{pt} \}$ denote a basis for $H^*(E)$ with $\langle \sigma_a, \sigma_b\rangle = 1$. Let $\psi, \chi$ denote anticommuting variables. An $\mc N = 2$ analytic superfield valued in $\mf g$ is a formal series in $\mf g[\![z, z^{-1}]\!][\chi, \psi]$. We do not yet impose any locality conditions on these fields and we do not impose invariance under superconformal transformations. We will consider the symbol $\oint (-)$ to denote $\frac{1}{2\pi i} \oint dz d\chi d\psi(-)$ where $\oint dz d\chi d\psi \psi\chi f(z) = \oint dz f(z)$. 
Consider the $\mc N = 2$ analytic superfield $w(z, \psi, \chi)$ with Fourier expansion
\begin{align*}
   w( z, \psi, \chi) &= \sum_{n \in \Z} w^{0,n}_{E}z^{-n} + \sum_{n\in \Z}w^{0,n}_{\sigma_a}z^{-n}\psi + \sum_{n\in \Z}w^{0,n}_{\sigma_b}z^{-n}\chi + \sum_{n\in \Z}w^{0,n}_{\rm{pt}}z^{-n}\psi\chi\\
&= w_{E}(z) + \psi w_{\sigma_a}(z) + \chi w_{\sigma_b}(z) + \psi\chi w_{\rm{pt}}(z). 
\end{align*}

More generally given $m \in \Z $ we define analytic superfields
\[W_m(z, \psi, \chi) = W_{m, E}(z) + \psi W_{m, \sigma_a}(z) +\chi W_{m, \sigma_b}(z) +\chi\psi W_{m, \rm{pt}}(z)  \] 
where 
\[ W_{m , \gamma}(z) = \sum_{n b\in \Z} w_{\gamma}^{m, n} z^{-n} \] 
so that $w= W_0$. 

Let $\End(V_{T^*E})_m$ denote the subspace of $\End(V_{T^*E})$ of operators $\phi$ such that $\phi(e^{kE}\otimes v) = e^{{k+m}E}\otimes v'$ for any $k$ and some $v'$. We may think of $W_m$ as sort of generalization of the $\mc N = 0$ vertex operator $\Gamma_{mE}(z)$ by considering the field 
\begin{align*}\label{eq:VO} \Gamma_m(z, \psi, \chi) &= : \exp (m (w(z, \psi, \chi) + E)): \\
  &= e^{mE} \left(1+ \psi \sum_{n \in \Z} z^{-n} \alpha_{n}(m\sigma_a) \right)\\
  & \left(1+ \chi \sum_{n \in \Z} z^{-n} \alpha_{n}(m\sigma_b) \right)
  :\left(1 +\psi\chi \sum_{n \in \Z} z^{-n} \alpha_n(m\rm{pt})\right)  \Gamma_{mE}(z):.
\end{align*}

This is an analogue of the $N = 1$ SUSY Lattice vertex operators of \cite{Heluani_Kac_2007} according to the formula
\[ \Gamma(z , \theta) = : \exp (\pmb{\phi}(\eta, z) + \theta\psi(\eta, z)  ): \]  
giving super vertex operators in a superconformal vertex algebra $V(L) \otimes F(\Pi  L\otimes \C)$ where $\psi(\eta, z)$ is a generating field for the free fermions  $F(\Pi  L\otimes \C)$. 

Then \eqref{eq:d_freefld}-\eqref{eq:sm_freefld} implies that up to 1) a rescaling of the coefficients and 2) the addition of terms involving the normal ordered product of Heisenberg fields and $\Gamma_{mE}(z)$ to the $\psi\chi$ coefficient of $\Gamma_{m(z, \psi, \chi)}$ we have an agreement between $\Gamma_{m, \psi, \chi}$ and $W_m(z, \psi, \chi)$. The correction terms indicate that the immediate generalization of the Kac-Heluani formula is not the appropriate formula in our context without modification.  It would be interesting to give a more satisfying explanation for the remaining terms in the language of $\mc N = 2$ analytic superfields and an attempt is given below. 
\subsubsection{}
We can also combine the generating series \eqref{eq:wgen} into two-variable formal series $\mathbb{Y}(\gamma, z_1, z_2) \in \mf g_{T^*E}[[z_1, z_2, z_1^{-1}, z_2^{-1}]]$ of the form 
\begin{equation}
\mathbb{Y}(\gamma, z_1, z_2) = \sum_{m,n\in \Z} \Upsilon_m(\gamma, z_1) z_2^{-m}. 
\end{equation}
with bracket 
\begin{multline}
\label{eq:two_variable_bracket}
[\mathbb{Y}(\gamma, z_1, z_2), \mathbb{Y}(\eta, w_1, w_2)] = \langle \gamma, \eta\rangle D_{w_1}\delta\left(\frac{w_1}{z_1}\right) \delta\left(\frac{w_2}{z_2}\right) \\
+ D_{w_1}\mathbb{Y}(\gamma\star\eta, w_1, w_2)\delta\left(\frac{w_1}{z_1}\right)D_{w_2}\delta\left(\frac{w_2}{z_2}\right)
- D_{w_2}\mathbb{Y}(\gamma\star\eta, w_1, w_2)D_{w_1}\delta\left(\frac{w_1}{z_1}\right)\delta\left(\frac{w_2}{z_2}\right) .
\end{multline}
Using the above analysis it is tempting to consider these as components of an analytic superfield 
\[ \mathbb{D}(z_1,z_2, \psi_+,  \psi_-)) := \mathbb{Y}(\pt, z_1,z_2) + \mathbb{Y}(\sigma_+,  z_1,z_2)\psi_- + \mathbb{Y}(\sigma_-,  z_1,z_2)\psi_+ + \mathbb{Y}(E,  z_1,z_2)\psi_+\psi_-\] 
of a $2|2$-valued coordinate $Z = (z; \psi) = (z_1, z_2; \psi_+, \psi_-) $ so that the bracket with a field at $W = (w; \psi) = (w_1, w_2; \phi_+, \phi_-)$ is given by 
\begin{align*}
  \label{eq:22_notation_D_bracket_local}
  [\mbb{D}(Z) ,\mbb{D}(W)] &= \frac 1 2 
  ( D_+D_{w_1} \mbb{D}(W)D_-D_{w_2} \delta -D_+ D_{w_2} \mbb{D}(W)D_- D_{w_1} \delta\\
  &+ D_-D_{w_1} \mbb{D}(W)D_+D_{w_2} \delta- D_-D_{w_2} \mbb{D}(W)D_+D_{w_1} \delta\\
  &+D_-D_+D_{w_1} \mbb{D}(W)D_{w_2} \delta- D_-D_+D_{w_2} \mbb{D}(W)D_{w_1} \delta\\
  &+  D_{w_1} \mbb{D}(W)D_{w_2}D_-D_+ \delta - D_{w_2} \mbb{D}(W)D_{w_1} D_-D_+\delta\\
  &+ D_-D_+D_{w_2} \mbb{D}(W)D_{w_1}\delta- D_-D_+D_{w_1} \mbb{D}(W)D_{w_2}\delta  ) \\
  &+ D_{w_1} \delta 
\end{align*}
where $D_\pm = \partial_{\psi_\pm}$ with the convention that differentials commute with the opposite set of variables, i.e. $D_\pm\phi_\pm = \phi_\pm D_\pm$ while $D_\pm\psi_\pm = -\psi_\pm D_{\pm}$. In the above formula
\[ \delta = \delta(W/Z) := \delta(w/z)\delta(\phi_+/\psi_+)\delta(\phi_-/\psi_-) \]  where  $\delta(\theta/\chi) = (\theta + \chi)$ for odd variables $\theta$ and $\chi$ and  $\delta(w/z) = \delta(w_1/z_1)\delta(w_2/z_2)$. If we do not require the right hand side to be expressed in terms of $\delta$, \eqref{eq:two_variable_bracket} is also equivalent to the commutator
\begin{equation}
  \label{eq:22_notation_D_bracket}
  [\mbb{D}(Z) ,\mbb{D}(W)] = D_{w_1}\mbb{D}(W + \psi)D_{w_2}\delta\left(\frac w z\right) -D_{w_2} \mbb{D}(W + \psi)D_{w_1}\delta\left(\frac w z\right) + D_{w_1} \delta\left( \frac W Z \right)
\end{equation}
where $W + \psi = (w_1, w_2; \phi_+ + \psi_+, \phi_- + \psi_-)$. 

\section{Loop direction and higher rank} 

\subsection{Tautological classes} 

Lehn \cite{Lehn_1999} has identified formulas for generators of the cohomology ring of $X_R^{[n]}$. 

In particular, for divisor classes we have 
\begin{equation}
  \label{eq:c1pt}
   c_1(\delta_{\pt})\cup - = -:\pmb{\alpha}^3:(z) [z^0] + (a - \frac \hbar 2 ) :\pmb{\alpha}^2:(z) [z^0] + \frac \hbar 2  \pmb{\Omega}
\end{equation}
where we have implicitly inserted the Kunneth components of the small diagonal,  $a \in H_T(\pt)$ acts trivially on the moduli space and 
\[ \pmb{\Omega} = \sum_{n > 0, i} n \alpha_{-n}(\gamma_i)\alpha_{n}(\gamma^i)\] 
is a non-local term which is not a Fourier coefficient of a vertex operator. The other divisor classes come from components of $L_0$ weighted by divisor classes on the surface; for example in the $A_{-1}$ case the other divisor takes the form 
\[ c_1(\delta_E)\cup - = -\frac 1 \hbar \sum \alpha_n(\pt) \alpha_{-n}(\pt) \] 
which is $L_0(\pt/\hbar)$. 

There are higher spin currents $\mf{L}^p(\gamma, z)$ forming a version of a $W_{1 + \infty}$ algebra acting on $V_{T^*E}$ \cite{Li_Qin_Wang_2002}. 

The commutation relations between the fields $\mc D_m(\gamma, z)$ and the higher spin currents may be calculated using the noncommutative Wick formula. 

\subsection{Higher rank} 
We mention a higher rank version of the Lie algebra construction of the paper which will  prove useful in the quantization of the solution of the classical Yang-Baxter equation satisfied by the toroidal extended affine Lie algebras. 

Consider the tensor product 

\[V_{T^*E}^{\Jac{}\otimes n} = V_{T^*E}^{\Jac{}}\otimes \cdots \otimes V_{T^*E}^{\Jac{}} .\] 
Then nullity-$(n+1)$ toroidal extended affine Lie algebra also admits vertex representations by a tensor product of affine vertex algebras, Heisenberg-Virasoro vertex algebras and sub-VOAs of hyperbolic lattice VOAs \cite{Billig_2007}. Let $\mf t^{(n+1)}(\mf{gl}_0)$ denote the nullity-$(n+1)$ toroidal extended affine Lie algebra of $\mf{gl}_0$ (with toroidal cocycle set to $\mu = 0$, as always in this paper). 

\subsubsection{Ghost representation of affine vertex algebra}

Let 
\[ b_i(z), c_i(z), ~~ i = 1, \ldots, n \]
denote generating fields of a rank $n$ $bc$-system (which plays a very different role from the higher genus rank $g$ system of Section \ref{ssec:SF}). Consider the fields 
\begin{equation}\label{eq:GL_ghost} 
  x_{ij}(z) = :c_i(z) b_j(z):
\end{equation}
A calculation using Wick's formula immediately gives 
\begin{prop}\label{prop:ghost_affine}
The assignment \eqref{eq:GL_ghost} gives a level 1 representation of $\widehat{\mf{gl}}_n$ on $V_{bc}^{\otimes n}$. 
\end{prop}
The generating fields of $\widehat{\mf{sl}}_n \subset \widehat{\mf{gl}}_n$ are primary of conformal dimension $1$ with respect to 
\[ \omega_{bc, n}(z) = :\partial c_1(z) b_1(z) : + \cdots :\partial c_n(z) b_n(z):\] 
while the Heisenberg field 
\[ J_n(z) = :c_1(z) b_1(z) : + \cdots : c_n(z) b_m(z):\] has central charge $k_H = n$ and the twisted central charge is $k_{HV} = n/2$.

\subsubsection{Higher rank toroidal extended affine representation} 
Using the ghost representation of $\widehat{\mf{gl}}_n$ we obtain the following from Billig's construction. 
\begin{thm}\label{thm:higher_rank}The vertex algebra 
  $V_{T^*E}^{\Jac{}\otimes n} $ admits a representation of $\mf t^{(n+1)}(\mf{gl}_0)$
\end{thm} 
\begin{proof}
We have a tensor decomposition 
\[ V_{T^*E}^{\Jac{} \otimes n} = V^+(\RN{2}_{1,1})^{\otimes n} \otimes V_{bc}^{\otimes n}.\] By Proposition \ref{prop:ghost_affine} we get a representation $V_1(\mf{sl}_n) \to V_{bc}^{\otimes n}$ and a twisted Heisenberg-Virasoro algebra from fields $\omega_{bc,n}(z)$ and $J_n(z)$ with $k_H = n, k_V = -2n,  k_{VH} = n/2$. This gives a representation of the toroidal extended affine Lie algebra by \cite{Billig_2007}. 
\end{proof} 
\begin{rmk}
By considering normal ordered products of free fermionic fields and vertex operators $Y(e^{a_1 E_1 + \cdots + a_n E_n}, z)$ we obtain a Lie superalgebra extending $t^{(n+1)}(\mf{gl}_0)$ as in the nullity-2 case. 
\end{rmk}

% \subsection{Superconformal structures} 

% \subsubsection{Superconformal algebra} 

% The vertex algebra $V_{T^*E}^{\Jac}$ possesses an action of the $N = 2$ superconformal algebra. This action is similar to the construction of the hidden $N = 2$ symmetry in the $N=1$ superstring \cite{Belavin_Spodyneiko_2015}. 

% The $N = 2$ Ramond superconformal algebra 
% is generated by the modes of fields 
% \begin{align*}
%    T(z)&& G^+(z)&& G^{-}(z) && J(z) \end{align*}
%    satisfying OPEs 
%    \begin{align*}
% J(z)J(w) &\sim \frac{c/3}{(z-w)^2}, & G^\pm(z)G^\pm(w)&  \sim 0 \\
% J(z)G^\pm(w) &\sim \pm \frac{G^pm(w)}{z-w} \\
% G^+(z) G^-(w) \sim \frac{c/3}{(z-w)^3} + \frac{J(w)}{(z-w)^2} + \frac{L(w) + 1/2 \partial J(w)}{z-w} 
%    \end{align*}
% such that $J(z)$ has conformal weight $1$ and $
%    There is an identification 
%    \[ \on{Modes}(V_{bc}) \simeq \on{Modes}(\{ \psi^+, \psi^-\})\] 
%    of associative algebras between the modes of the bc system and that of the charged free fermion. 

%   \begin{prop}
%   Under the identification 
%   \[ V_{T^*E}^{\Jac} = V^+(\RN{2}_{1,1})\otimes V_{bc} \]
%   the assignment \eqref{eq:jacT}-\eqref{eq:jacJ} gives a representation of the $N=2$ superconformal algebra on $V_{T^*E}^{\Jac}$. 
%   \end{prop}
%   \begin{proof}
%    Follows from Wick's formula. 
%   \end{proof}
\printbibliography

@book{bartocci2009fourier,
  title={Fourier-Mukai and Nahm transforms in geometry and mathematical physics},
  author={Bartocci, Claudio and Bruzzo, Ugo and Ruip{\'e}rez, Daniel Hern{\'a}ndez},
  volume={276},
  year={2009},
  publisher={Springer Science \& Business Media}
}

@inproceedings{simpson1990nonabelian,
  title={Nonabelian hodge theory},
  author={Simpson, Carlos T},
  booktitle={Proceedings of the International Congress of Mathematicians},
  volume={1},
  pages={747--756},
  year={1990}
}

@article{Maulik_Okounkov_2019, 
   title={Quantum  Groups and Quantum Cohomology}, 
   ISSN={0303-1179}, DOI={10.24033/ast.1074}, 
   number={408}, 
   journal={Asterisque}, 
   author={Maulik, Davesh and Okounkov, Andrei}, 
   year={2019}
 }

@book{Nakajima_1999,
 place={Providence,
 Rhode Island},
 series={University Lecture Series},
 title={Lectures on Hilbert Schemes of Points on Surfaces},
 volume={18},
 ISBN={978-0-8218-1956-2},
 DOI={10.1090/ulect/018},
 publisher={American Mathematical Society},
 author={Nakajima,
 Hiraku},
 year={1999},
 month={Sep},
 collection={University Lecture Series} }

@article{Grojnowski_1996,
 title={Instantons and affine algebras I: The Hilbert scheme and vertex operators},
 volume={3},
 ISSN={10732780,
 1945001X},
 DOI={10.4310/MRL.1996.v3.n2.a12},
 number={2},
 journal={Mathematical Research Letters},
 author={Grojnowski,
 I.},
 year={1996},
 pages={275–291} }

@misc{rains19birational_arxiv,
  url = {https://arxiv.org/abs/1907.11301},
  author = {Rains, Eric M.},
  title = {The birational geometry of noncommutative surfaces},
  publisher = {arXiv},
  year = {2019},
  copyright = {arXiv.org perpetual, non-exclusive license}
}

@article{vdBergh_2012, 
 title={Non-commutative {P1}-bundles over commutative schemes}, 
 volume={364}, 
 ISSN={0002-9947}, 
 number={12}, 
 journal={Transactions of the American Mathematical Society}, 
 author={Van den Bergh, Michel}, 
 year={2012}, 
 pages={6279–6313} }

@article{ben2008perverse,
  title={Perverse bundles and Calogero--Moser spaces},
  author={Ben-Zvi, David and Nevins, Thomas},
  journal={Compositio Mathematica},
  volume={144},
  number={6},
  pages={1403--1428},
  year={2008},
  publisher={London Mathematical Society}
}

@article{Bayer_Lahoz_Macrì_Nuer_Perry_Stellari_2021, title={Stability conditions in families}, volume={133}, ISSN={1618-1913}, number={1}, journal={Publications mathématiques de l’IHÉS}, publisher={Springer}, author={Bayer, Arend and Lahoz, Martí and Macrì, Emanuele and Nuer, Howard and Perry, Alexander and Stellari, Paolo}, year={2021}, pages={157–325} }

@article{Kassel_1984, title={K{\"a}hler differentials and coverings of complex simple Lie algebras extended over a commutative algebra}, volume={34}, ISSN={0022-4049}, number={2–3}, journal={Journal of Pure and Applied Algebra}, publisher={North-Holland}, author={Kassel, Christian}, year={1984}, pages={265–275} }

@article{Chen_Li_Tan_2021, title={Toroidal extended affine Lie algebras and vertex algebras}, journal={arXiv preprint arXiv:2102.10968}, author={Chen, Fulin and Li, Haisheng and Tan, Shaobin}, year={2021} }

@article{Billig_2006, title={A category of modules for the full toroidal Lie algebra}, volume={2006}, ISSN={1073-7928}, number={9}, journal={International Mathematics Research Notices}, publisher={OUP}, author={Billig, Yuly}, year={2006}, pages={68395–68395} }

@article{Billig_2007, title={Representations of toroidal extended affine Lie algebras}, volume={308}, ISSN={0021-8693}, number={1}, journal={Journal of Algebra}, publisher={Elsevier}, author={Billig, Yuly}, year={2007}, pages={252–269} }

@article{Moody_Rao_Yokonuma_1990, title={Toroidal Lie algebras and vertex representations}, volume={35}, ISSN={0046-5755}, number={1–3}, journal={Geometriae dedicata}, publisher={Springer}, author={Moody, Robert V and Rao, Senapathi Eswara and Yokonuma, Takeo}, year={1990}, pages={283–307} }

@book{Huybrechts_2006, title={Fourier-Mukai transforms in algebraic geometry}, ISBN={0-19-151635-X}, publisher={Clarendon Press}, author={Huybrechts, Daniel}, year={2006} }

@article{Saito_Takebayashi_1997, title={Extended affine root systems III (elliptic Weyl groups)}, volume={33}, ISSN={0034-5318}, number={2}, journal={Publications of the Research Institute for Mathematical Sciences}, author={Saito, Kyoji and Takebayashi, Tadayoshi}, year={1997}, pages={301–329} }

@article{Lehn_1999, title={Chern classes of tautological sheaves on Hilbert schemes of points on surfaces}, volume={136}, ISSN={1432-1297}, DOI={10.1007/s002220050307}, abstractNote={We give an algorithmic description of the action of the Chern classes of tautological bundles on the cohomology of Hilbert schemes of points on a smooth surface within the framework of Nakajima’s oscillator algebra. This leads to an identification of the cohomology ring of Hilbn(A2) with a ring of explicitly given differential operators on a Fock space. We end with the computation of the top Segre classes of tautological bundles associated to line bundles on Hilbnup to n=7, extending computations of Severi, LeBarz, Tikhomirov and Troshina and give a conjecture for the generating series.}, number={1}, journal={Inventiones mathematicae}, author={Lehn, Manfred}, year={1999}, month={Mar}, pages={157–207}, language={en} }

@article{Kausch_2000, title={Symplectic fermions}, volume={583}, ISSN={0550-3213}, number={3}, journal={Nuclear Physics B}, publisher={Elsevier}, author={Kausch, Horst G}, year={2000}, pages={513–541} }

@article{Kausch_1995, title={Curiosities at c=-2}, journal={arXiv preprint hep-th/9510149}, author={Kausch, Horst G}, year={1995} }

@article{Groechenig_2014, title={Hilbert schemes as moduli of Higgs bundles and local systems}, volume={2014}, ISSN={1073-7928}, number={23}, journal={International Mathematics Research Notices}, publisher={OUP}, author={Groechenig, Michael}, year={2014}, pages={6523–6575} }

@book{Kac_1998, title={Vertex algebras for beginners}, ISBN={0-8218-1396-X}, number={10}, publisher={American Mathematical Soc.}, author={Kac, Victor G}, year={1998} }

@book{Frenkel_Lepowsky_Meurman_1989, title={Vertex operator algebras and the Monster}, ISBN={0-08-087454-1}, publisher={Academic press}, author={Frenkel, Igor and Lepowsky, James and Meurman, Arne}, year={1989} }

@article{Chen_Li_Yu_2022, title={Irreducible modules of toroidal Lie algebras arising from {$\phi_\epsilon$}-coordinated modules of vertex algebras}, volume={611}, ISSN={0021-8693}, journal={Journal of Algebra}, publisher={Elsevier}, author={Chen, Fulin and Li, Huansheng and Yu, Nina}, year={2022}, pages={110–148} }

@article{Heluani_Kac_2007, title={SUSY Lattice Vertex Algebras}, url={http://arxiv.org/abs/0710.1587}, abstractNote={We construct and study SUSY lattice vertex algebras. As a simple example, we obtain the simple vertex algebra associated to the vertex algebra Vc(N 3) of central charge c = 3/2, as the SUSY lattice vertex algebra associated to Z with bilinear form (a, b) = 2ab.}, note={arXiv:0710.1587 [math]}, number={arXiv:0710.1587}, publisher={arXiv}, author={Heluani, Reimundo and Kac, Victor G.}, year={2007}, month={Oct}, language={en} }

@article{Bayer_Macrì_2014_MMP, title={{MMP} for moduli of sheaves on K3s via wall-crossing: nef and movable cones, Lagrangian fibrations}, volume={198}, ISSN={0020-9910}, number={3}, journal={Inventiones mathematicae}, publisher={Springer}, author={Bayer, Arend and Macr{\`i'}, Emanuele}, year={2014}, pages={505–590} }

@article{Li_2011, title={{$\phi$}-Coordinated quasi-modules for quantum vertex algebras}, volume={308}, ISSN={0010-3616}, number={3}, journal={Communications in mathematical physics}, publisher={Springer}, author={Li, Haisheng}, year={2011}, pages={703–741} }

@article{Lepowsky_2000, title={Application of a ‘Jacobi Identity’for vertex operator algebras to zeta values and differential operators}, volume={53}, ISSN={0377-9017}, journal={Letters in Mathematical Physics}, publisher={Springer}, author={Lepowsky, James}, year={2000}, pages={87–103} }

@book{Zhu_1990, title={Vertex operator algebras, elliptic functions and modular forms}, ISBN={9798207589855}, publisher={Yale University}, author={Zhu, Yongchang}, year={1990} }

@article{Mukai_1981, title={Duality between D (X) and with its application to Picard sheaves}, volume={81}, ISSN={0027-7630}, journal={Nagoya Mathematical Journal}, publisher={Cambridge University Press}, author={Mukai, Shigeru}, year={1981}, pages={153–175} }

@article{Eswara_Rao_Moody_1994, title={Vertex representations for n-toroidal Lie algebras and a generalization of the Virasoro algebra}, volume={159}, ISSN={0010-3616}, number={2}, journal={Communications in mathematical physics}, publisher={Springer}, author={Eswara Rao, S and Moody, RV}, year={1994}, pages={239–264} }

@article{Nakajima_1996_jack, title={Jack polynomials and Hilbert schemes of points on surfaces}, journal={arXiv preprint alg-geom/9610021}, author={Nakajima, Hiraku}, year={1996} }

@article{Barron_2010, title={Axiomatic aspects of N= 2 vertex superalgebras with odd formal variables}, volume={38}, ISSN={0092-7872}, number={4}, journal={Communications in Algebra}, publisher={Taylor & Francis}, author={Barron, Katrina}, year={2010}, pages={1199–1268} }

@article{Gurarie_1993, title={Logarithmic operators in conformal field theory}, volume={410}, ISSN={0550-3213}, number={3}, journal={Nuclear Physics B}, publisher={Elsevier}, author={Gurarie, Victor}, year={1993}, pages={535–549} }

@article{Bakalov_Villarreal_2022, title={Logarithmic vertex algebras}, ISSN={1083-4362}, journal={Transformation Groups}, publisher={Springer}, author={Bakalov, Bojko N and Villarreal, Juan J}, year={2022}, pages={1–63} }

@article{Nagao_2007, title={Quiver varieties and Frenkel-Kac construction}, url={http://arxiv.org/abs/math/0703107}, abstractNote={An aﬃne Lie algebra acts on cohomology groups of quiver varieties of aﬃne type. A Heisenberg algebra acts on cohomology groups of Hilbert schemes of points on a minimal resolution of a Kleinian singularity. We show that in the case of type A the former is obtained by Frenkel-Kac construction from the latter.}, note={arXiv: math/0703107}, journal={arXiv:math/0703107}, author={Nagao, Kentaro}, year={2007}, month=mar, language={en} }

@article{DeHority_2020, title={Affinizations of Lorentzian Kac-Moody Algebras and Hilbert Schemes of Points on K3 Surfaces}, journal={arXiv preprint arXiv:2007.04953}, author={DeHority, Samuel}, year={2020} }

@article{Nakajima_1998, title={Quiver varieties and Kac-Moody algebras}, volume={91}, ISSN={0012-7094}, DOI={10.1215/S0012-7094-98-09120-7}, number={3}, journal={Duke Mathematical Journal}, author={Nakajima, Hiraku}, year={1998}, month=feb, pages={515–560}, language={en} }

@article{dehority_toroidal_springer, 
title = {Toroidal analogues of the Grothendieck-Springer map}, 
author = {Samuel De{H}ority},
note = {In preparation}
}

@article{dehority_quatum_connection, 
title = {Quantum connection of some isotrivial lagrangian fibrations}, 
author = {Samuel De{H}ority},
note = {In preparation}
}

@article{Namikawa_2015, title={Poisson Deformations and Birational Geometry}, volume={22}, journal={J. Math. Sci. Univ. Tokyo}, author={Namikawa, Yoshinori}, year={2015}, pages={339–359} }

@article{Boalch_2018, title={Wild character varieties, meromorphic Hitchin systems and Dynkin diagrams}, volume={2}, journal={Geometry and physics}, publisher={Oxford Univ. Press, Oxford}, author={Boalch, Philip}, year={2018}, pages={433–454} }

@article{Liu_Lo_Martinez_2019, title={Fourier-Mukai transforms and stable sheaves on Weierstrass elliptic surfaces}, journal={arXiv preprint arXiv:1910.02477}, author={Liu, Wanmin and Lo, Jason and Martinez, Cristian}, year={2019} }

@article{Lo_Martinez_2022, title={Geometric stability conditions under autoequivalences and applications: Elliptic Surfaces}, journal={arXiv preprint arXiv:2210.01261}, author={Lo, Jason and Martinez, Cristian}, year={2022} }

@article{Yoshioka_2022, title={A note on stability conditions on an elliptic surface}, journal={arXiv preprint arXiv:2211.08079}, author={Yoshioka, Kota}, year={2022} }

@book{Carmeli_Caston_Fioresi_2011, title={Mathematical foundations of supersymmetry}, volume={15}, ISBN={3-03719-097-3}, publisher={European Mathematical Society}, author={Carmeli, Claudio and Caston, Lauren and Fioresi, Rita}, year={2011} }

@article{Saito_1985, title={Extended affine root systems I (Coxeter transformations)}, volume={21}, ISSN={0034-5318}, number={1}, journal={Publications of the Research Institute for Mathematical Sciences}, publisher={Research Institute forMathematical Sciences}, author={Saito, Kyoji}, year={1985}, pages={75–179} }

@article{Burban_Schiffmann_2013, title={The composition Hall algebra of a weighted projective line}, volume={2013}, ISSN={1435-5345}, number={679}, journal={Journal für die reine und angewandte Mathematik (Crelles Journal)}, publisher={De Gruyter}, author={Burban, Igor and Schiffmann, Olivier}, year={2013}, pages={75–124} }

@article{Bridgeland_1998, title={Fourier-Mukai transforms for elliptic surfaces}, volume={1998}, number={498}, journal={Journal für die reine und angewandte Mathematik (Crelles Journal)}, publisher={Walter de Gruyter GmbH}, author={Bridgeland, Tom}, year={1998}, pages={115–133} }

@article{Creutzig_Ridout_2013, title={Logarithmic conformal field theory: beyond an introduction}, volume={46}, ISSN={1751-8121}, number={49}, journal={Journal of Physics A: Mathematical and Theoretical}, publisher={IOP Publishing}, author={Creutzig, Thomas and Ridout, David}, year={2013}, pages={494006} }

@article{Li_Qin_Wang_2002, title={Hilbert schemes and W-algebras}, volume={2002}, ISSN={1073-7928}, number={27}, journal={International Mathematics Research Notices}, publisher={OUP}, author={Li, Wei-Ping and Qin, Zhenbo and Wang, Weiqiang}, year={2002}, pages={1427–1456} }

@article{Costello_Gaiotto_2021, title={Twisted Holography}, url={http://arxiv.org/abs/1812.09257}, abstractNote={We derive and test a novel holographic duality in the B-model topological string theory. The duality relates the B-model on certain Calabi-Yau three-folds to twodimensional chiral algebras deﬁned as gauged βγ systems. The duality conjecturally captures a topological sector of more familiar AdS5/CFT4 holographic dualities.},  number={arXiv:1812.09257}, publisher={arXiv}, author={Costello, Kevin and Gaiotto, Davide}, year={2021}, month=jan, language={en} }

@article{Hausel_Mellit_Minets_Schiffmann_2022, title={$ P= W $ via $ H_2$}, journal={arXiv preprint arXiv:2209.05429}, author={Hausel, Tamás and Mellit, Anton and Minets, Alexandre and Schiffmann, Olivier}, year={2022} }

@article{Bittleston_2023, title={On the associativity of 1-loop corrections to the celestial operator product in gravity}, volume={2023}, ISSN={1029-8479}, number={1}, journal={Journal of High Energy Physics}, publisher={Springer}, author={Bittleston, Roland}, year={2023}, pages={1–59} }
\end{document}